\documentclass{amsart}
\usepackage{enumerate,enumitem}
\usepackage{amsthm, amssymb, amsmath}
\usepackage[initials, non-compressed-cites]{amsrefs}
\usepackage{graphicx, tikz}
\usetikzlibrary{calc,patterns,angles, quotes}
\usepackage{mathrsfs}
\usepackage{mathtools}
\usepackage{subfig}
\usepackage{tabularx}
\usepackage{fancyhdr}
\usepackage{xcolor}
\usepackage{hyperref}

\hypersetup{
    colorlinks,
    linkcolor={red!75!black},
    citecolor={blue!70!black},
    urlcolor={blue!80!black}
}

\newtheorem{theorem}{Theorem}[section]

\newtheorem{cor}[theorem]{Corollary}
\newtheorem*{theo*}{Theorem}
\newtheorem*{cor*}{Corollary}
\newtheorem{lm}[theorem]{Lemma}

\theoremstyle{definition}

\newtheorem*{rem*}{Remark}

\newtheorem{ex}[theorem]{Example}

\newcommand{\D}{\mathbb{D}}
\newcommand{\C}{\mathbb{C}}
\newcommand{\R}{\mathbb{R}}
\newcommand{\N}{\mathbb{N}}
\renewcommand{\H}{\mathbb{H}}

\renewcommand{\epsilon}{\varepsilon}

\numberwithin{equation}{section}

\title[Monotonicity of projections]{Monotonicity properties of hyperbolic projections in holomorphic iteration}
\author[A. Christodoulou]{Argyrios Christodoulou}
\address{Department of Mathematics, Aristotle University of Thessaloniki, 54124, Thessaloniki, Greece}
\email{argyriac@math.auth.gr}
\author[K. Zarvalis]{Konstantinos Zarvalis$^1$}
\address{Department of Mathematics, Aristotle University of Thessaloniki, 54124, Thessaloniki, Greece}
\email{zarkonath@math.auth.gr}
\subjclass[2020]{Primary: 37F44, 30F45; Secondary: 30D05, 51M10}
\keywords{Holomorphic iteration; hyperbolic geometry; hyperbolic projection}
\thanks{$^1$ Partially supported by Junta de Andaluc\'{i}a, grant number QUAL21 005 USE}
\begin{document}
	
	\begin{abstract}
		We consider hyperbolic projections of orbits of holomorphic self-maps of the unit disc, onto curves landing on the unit circle with a given angle. We show that under certain, necessary, assumptions, the projections exhibit monotonicity properties akin to those present in continuous dynamics. Our techniques are purely hyperbolic-geometric in nature and provide the general framework for analysing projections of arbitrary sequences onto curves.
	\end{abstract}
	
	\maketitle

\section{Introduction}
This article explores the monotonicity properties of hyperbolic projections of holomorphic orbits in the unit disc of the complex plane, onto curves landing on the unit circle. Our analysis is inspired by the theory of continuous holomorphic semigroups, where examining the behaviour of projections of trajectories has been at the forefront of many advancements in the last decade \cites{BK, Bracci-Speeds, BCK, Cordella1, Cordella2, KZ, Zar-Tangential}. In fact, our main result, Theorem \ref{thm: main result}, can be thought of as a strong, discrete version of the remarkable monotonicity property for the trajectories of continuous semigroups proved recently by Betsakos and Karamanlis \cite{BK}. Surprisingly, this strong monotonicity is unique both to discrete iteration and to the type of functions considered in Theorem \ref{thm: main result}, as shown by our examples in Section \ref{sect: examples}. 

The results presented here are part of a host of recent advances in the theory of discrete iteration in the unit disc that aim to shed light into the precise nature with which a sequence of iterates converges to a boundary attracting fixed point. These include a new version of the Denjoy--Wolff Theorem for arbitrary simply connected domains \cite{BB}, an analysis of the slope problem in parabolic iteration \cites{CCZRP 1, CCZRP 2}, and purely geometric versions of classical results \cite{Beardon-Minda}, to name a few. Also, of particular interest are the novel connections between discrete iteration and other areas of holomorphic dynamics, such as continuous semigroups \cites{Bracci-Roth, CGDM}, and dynamics of entire functions \cites{Evdoridou, BFJK}.

\medskip

Let us start by writing $\D$ for the open unit disc in the complex plane $\C$, which we endow with the hyperbolic metric $d_\D$. The classical Denjoy--Wolff Theorem \cite{Wolff} states that if $f\colon\D\to\D$ is a holomorphic function with no fixed points in $\D$, then there exists a unique point $\tau\in\partial\D$, called the \emph{Denjoy--Wolff point of $f$}, so that the sequence of iterates of $f$,
\[
f^n\vcentcolon=\underbrace{f\circ f \circ \cdots \circ f}_{n \text{ times}},
\]
converges uniformly on compact sets to the constant function $\tau$. Moreover, the \emph{angular derivative} $ f'(\tau)$ of $f$ at $\tau$ exists and satisfies $f'(\tau)\in(0,1]$. When $f'(\tau)<1$ the holomorphic function $f$ is called \emph{hyperbolic}, and \emph{parabolic} otherwise.

We say that a curve $\gamma\colon [0,+\infty)\to\D$ \emph{lands} at a point $\zeta\in\partial\D$, if $\gamma(t)$ converges to $\zeta$ as $t\to +\infty$, in the Euclidean metric of $\C$. Given $z_0\in\D$, a point $\pi_\gamma(z_0)\in\D$ will be called \emph{a projection of $z_0$ onto $\gamma$} if $\pi_\gamma(z_0)=\gamma(t_0)$, for some $t_0\geq0$, and
\[
d_\D(z_0,\gamma(t_0))=\inf\{d_\D(z_0,\gamma(t))\colon t\in[0,+\infty)\}.
\]
Note that when $\gamma$ lands at a point on the boundary, projections of any point $z\in\D$ onto $\gamma$ always exist, but are not necessarily unique. It is well-known and easy to prove, however, that when $\gamma$ is a geodesic of the hyperbolic metric, every point of the disc has a unique projection onto $\gamma$.

\medskip

Our main objective is to establish the following theorem:

\begin{theorem}\label{thm: main result}
    Let $f\colon\D\to\D$ be a hyperbolic map with Denjoy--Wolff point $\tau$, and $\gamma\colon [0,+\infty)\to\D$ a curve landing at $\tau$, satisfying
    \[
    \lim_{t\to+\infty}\mathrm{arg}(1-\overline{\tau}\gamma(t))=\theta\in\left(-\tfrac{\pi}{2},\tfrac{\pi}{2}\right).
    \]
    Fix $z\in\D$ and let $\pi_\gamma(f^n(z))$ be a projection of $f^n(z)$ onto $\gamma$, for $n\in\N$. Then, for any $w\in\D$, the sequence $\{d_\D(w,\pi_\gamma(f^n(z)))\}$ is eventually strictly increasing.
\end{theorem}

For a curve $\gamma\colon[0,+\infty)\to\D$ landing at $\tau\in\partial\D$, the cluster set of $\mathrm{arg}(1-\overline{\tau}\gamma(t))$ is often called the \emph{slope} of the curve. So in Theorem \ref{thm: main result} we essentially assume that $\gamma$ lands on $\tau$ non-tangentially (i.e. stays within a Stolz angle) and with a specified slope. The archetypal examples of such curves are the hyperbolic geodesics of $\D$ landing at $\tau$, for which $\mathrm{arg}(1-\overline{\tau}\gamma(t))$ tends to zero. Therefore, as a special case of Theorem \ref{thm: main result}, we have that the distance between any point $w\in\D$ and the projections of a hyperbolic orbit $\{f^n(z)\}$ onto any geodesic is eventually strictly increasing (see Corollary \ref{cor:orthogonal speed corollary}).

The main advantage of our result, however, is the fact that it allows for projections onto \emph{any} curve landing at the Denjoy--Wolff point with a specified slope. This may even include behaviours typically considered ``pathological", such as the curve having (diminishing) oscillations or self-intersections.

Moreover, Theorem \ref{thm: main result} allows for complete freedom for the choice of each projection $\pi_\gamma(f^n(z))$, of which there might be infinitely many. Finally, it is worth mentioning that previous monotonicity results of this type---to be described shortly---consider the distance of the projections $\pi_\gamma(f^n(z))$ from $z$, i.e. the case $w=z$ in Theorem \ref{thm: main result}. The fact that we can choose the base-point $w$ arbitrarily encapsulates the intuitive fact of hyperbolic geometry that an ``observer" positioned at $w\in\D$ will eventually perceive the sequence $\{\pi_\gamma(f^n(z))\}$ as moving away from them monotonically, regardless of the position of $w$.

\medskip

Our proof is based on a thorough analysis of the hyperbolic geometry behind projections of sequences onto curves, carried out in Section \ref{sect:projections}. The main technical result of this work is Theorem \ref{lm:orthogonal lemma 2}, which roughly states that if two curves $\gamma_1$ and $\gamma_2$ land at the same boundary point, with the same slope, then the projections $\pi_{\gamma_1}(z_n)$ and $\pi_{\gamma_2}(z_n)$ of a sequence $\{z_n\}\subset\D$ onto $\gamma_1$ and $\gamma_2$, respectively, are asymptotically close to one another. This result might be of independent interest since it implies that the projections onto an arbitrary curve can be controlled by projections onto simple, well-behaved curves landing with the same slope. In fact, we employ this technique for the proof of Theorem \ref{thm: main result} in Section \ref{sect: monotonicity}. Also, this result allows us to asymptotically correlate the projections onto curves landing with different angles, as shown in Corollary \ref{cor:projections with different slopes}.

\medskip

In order to describe the connection between Theorem \ref{thm: main result} and the theory of continuous semigroups of holomorphic functions, let us consider a semigroup $\phi_t\colon \D\to\D$, $t\geq0$, with Denjoy--Wolff point $\tau\in\partial\D$ (see the discussion in Example \ref{ex: hyperbolic semigroup} for the precise definition). Also take a geodesic $\gamma\colon [0,+\infty)\to\D$ of the hyperbolic metric, emanating from a point $z\in\D$ and landing at $\tau$.

The continuous function $v^z\colon [0,+\infty)\to[0,+\infty)$ with
\[
v^z(t)= d_\D(z,\pi_\gamma(\phi_t(z))),
\]
is called the \emph{orthogonal speed} of the semigroup. This was first considered by Bracci in his seminal paper \cite{Bracci-Speeds}, where he proved several key properties of $v^z$ and computed the rate with which $v^z(t)$ tends to infinity in various cases. His work has been the starting point for a plethora of modern results that investigate the properties of the speeds of convergence for the trajectories of semigroups, such as \cites{BCK,Cordella1, Cordella2, KZ, Zar-Tangential}.

Recently, Betsakos and Karamanlis \cite[Theorem 1.1]{BK} showed that the orthogonal speed is a strictly increasing function, for any semigroup with a boundary Denjoy--Wolff point. This result serves as the main motivation behind Theorem \ref{thm: main result}, since the sequence $\{d_\D(w,\pi_\gamma(f^n(z)))\}$ considered in Theorem~\ref{thm: main result} is an analogue of the orthogonal speed in discrete iteration. In fact, since our theorem allows for projections onto curves landing at $\tau$ non-tangentially, $\{d_\D(w,\pi_\gamma(f^n(z)))\}$ can be thought of as the \emph{non-tangential speed} of the map $f$. 

In Section \ref{sect: examples} we provide examples which show that the strict monotonicity described by Theorem~\ref{thm: main result} may fail for semigroups and for parabolic maps, meaning that it is unique to hyperbolic self-maps of $\D$. Moreover, we show that the term ``\emph{eventually}" cannot be omitted, and that the landing slope $\theta$ of the curve $\gamma$ cannot be considered to be $\pm\tfrac{\pi}{2}$. Hence, when considering the strict monotonicity of the non-tangential speed, our result is the best possible. 

\medskip

As an immediate corollary of Theorem \ref{thm: main result}, we obtain the following global monotonicity for hyperbolic maps. 

\begin{cor}\label{cor: main cor}
    Let $f\colon\D\to\D$ be a hyperbolic map with Denjoy--Wolff point $\tau$. For any $z,w\in\D$, the sequence $\{d_\D(w,f^n(z))\}$ is eventually strictly increasing.
\end{cor}

In the language of continuous semigroups introduced by Bracci in \cite{Bracci-Speeds}, the sequence considered in Corollary \ref{cor: main cor} is called \emph{total speed}, whenever $w=z$. In \cite[Theorem 1.2]{BK} it was shown that there exist semigroups for which the total speed is not eventually increasing. In Example \ref{ex:corollary counterexample} we describe how \cite[Theorem 1.2]{BK} can be used in order to show that the sequence $\{d_\D(w,f^n(z))\}$ may not be eventually increasing if $f$ is parabolic, and so the assumption that $f$ is hyperbolic in Corollary \ref{cor: main cor} is necessary. 

\medskip

In the sixth, and final, section of this article, we discuss questions that arise from our work as well as possible generalisations of our results. 

\section{Preliminaries}\label{sect: prelim}

\subsection{Hyperbolic Geometry}

We start by presenting the basics of hyperbolic geometry. For an in depth treatise on the subject, we refer to the excellent books \cite[Chapter 1]{Abate}, \cite[Chapter 7]{Beardon} and \cite[Chapter 5]{BCDM-Book}.

The \emph{hyperbolic metric} of $\D$ is defined as
\begin{equation*}
	d_\D(z,w)=\inf\limits_{\gamma}\int\limits_{\gamma}\frac{|dz|}{1-|z|^2},
\end{equation*}
where the infimum is taken over all piecewise $C^1$-smooth curves $\gamma:[0,1]\to\D$ satisfying $\gamma(0)=z$ and $\gamma(1)=w$. It is easy to show that $d_\D$ is in fact a metric, and that the metric space $(\D,d_\D)$ is complete. 

The Riemann mapping theorem allows us to transfer the above notions to any simply connected domain as follows. Let $\Omega\subsetneq\C$ be a simply connected domain and $f:\Omega\to\D$ a Riemann mapping. The hyperbolic distance in $\Omega$ is defined by
\begin{equation}\label{eq:hyperbolic distance in domain}
	d_\Omega(z,w)\vcentcolon=d_\D(f(z),f(w)), \quad z,w\in\Omega.
\end{equation}
It turns out that this definition is independent of the choice of the Riemann mapping $f$. More importantly, this immediately shows that the hyperbolic metric is a conformally invariant quantity. 

As final pieces of notation, we use $D_\Omega(w,R)$ to describe the open hyperbolic disc of $\Omega$, centred at $w$ and of radius $R$; that is the set
\begin{equation}\label{eq:hyperbolic disk}
	D_\Omega(w,R)\vcentcolon=\{z\in\Omega:d_\Omega(z,w)<R\}.
\end{equation}

Also, if $\gamma\colon[0,+\infty)\to\Omega$ is a curve and $z\in\Omega$, for convenience we write
\begin{equation}\label{eq: distance from curve}
    d_\Omega(z,\gamma)\vcentcolon=\inf\{d_\Omega(z,\gamma(t))\colon t\in[0,+\infty)\}.
\end{equation}

\subsection{The right half-plane}

The conformal invariance of the hyperbolic metric allows us to transfer our arguments from the unit disc to any simply connected domain. In fact, using the Cayley transform $T\colon \D\to\H$ with $T(z)=\frac{1+z}{1-z}$, we can pass to the right half-plane $\H=\{w\in\C:\textup{Re}w>0\}$. Working in $\H$ instead of $\D$ simplifies many of the computations involving the hyperbolic metric, and we will do so for the rest of this article.

For any $z,w\in\H$, the hyperbolic metric of $\H$ is given by the formula
\begin{equation}\label{eq:hyperbolic distance in H}
    d_\H(z,w)=\frac{1}{2}\log\frac{1+\left\lvert \frac{z-w}{z+\overline{w}}\right\rvert}{1-\left\lvert \frac{z-w}{z+\overline{w}}\right\rvert}.
\end{equation}

The quantity 
\begin{equation}\label{eq: pseudo-hyperbolic metric}
    \rho_\H(z,w)= \left\lvert \frac{z-w}{z+\overline{w}}\right\rvert, \quad \text{for } z,w\in\H,
    \end{equation}
is called the \emph{pseudo-hyperbolic metric} of $\H$. The pseudo-hyperbolic metric is indeed a metric on $\H$, with $\rho_\H(z,w)\in[0,1)$ for all $z,w\in\H$. Simple calculations show that $\rho_\H$ satisfies the following handy formula
\begin{equation}\label{eq: useful pseudo-hyp formula}
    1-\rho_\H(z,w)^2=\frac{4\ \mathrm{Re}z \ \mathrm{Re}w }{\lvert z+\overline{w}\rvert^2}, \quad \text{for any } z,w\in\H.
\end{equation}

Another useful formula, which can be easily obtained from \eqref{eq:hyperbolic distance in H} and \eqref{eq: useful pseudo-hyp formula}, states that for any $z,w\in\H$, we have 
\begin{equation}\label{eq:cosh in H}
    \cosh d_\H(z,w)= \frac{\lvert z+\overline{w}\rvert}{2\ \sqrt{\mathrm{Re} \ z}\ \sqrt{\mathrm{Re} \ w}}.
\end{equation}

The hyperbolic geodesics in $\H$ are the Euclidean straight half-lines and semicircles lying in $\H$ that are perpendicular to the imaginary axis.

\medskip

Let us now describe how our main objects of study are modified when moving to the right half-plane. First, any holomorphic function $f\colon \D \to \D$ induces a holomorphic function $F\colon \H \to \H$, by defining $F= T\circ f \circ T^{-1}$, where $T$ is the Cayley transform mentioned above. When $f$ has Denjoy--Wolff point $\tau\in\partial\D$, we can pre-compose $T$ with a rotation of $\D$ in order to obtain that the Denjoy--Wolff point of $F$ is the point at infinity. That is, the sequence of iterates $F^n$ of $F$ converges uniformly on compact sets to the point at infinity, in the Euclidean metric. Moreover, $F$ has an angular derivative at infinity, given by 
\[
F'(\infty)=\frac{1}{f'(\tau)}\in[1,+\infty).
\]
Hence, the holomorphic map $F$ is hyperbolic if $F'(\infty)>1$ and parabolic otherwise. Furthermore, by the Julia--Carath\'eodory Theorem \cite[Proposition 2.3.1]{Abate} we have that
\begin{equation}\label{eq: angular derivative}
    F'(\infty)=\angle\lim_{z\to\infty}\frac{F(z)}{z},
\end{equation}
where the angular notation in the limit simply means that $z$ tends to $\infty$ through a sector $\{z\in\C\colon \lvert \mathrm{arg}\ z\rvert<\phi\}$, for some $\phi\in(-\tfrac{\pi}{2},\tfrac{\pi}{2})$. For more information on the angular derivative, we refer to \cite[Chapter 4]{Pommerenke} and \cite[Section 2.3]{Abate}.

Suppose, now, that $\gamma\colon[0,+\infty)\to\D$ is a curve in $\D$, landing at a point $\zeta\in\partial\D$. Then, again using an appropriate version of $T$, we have that the induced curve $\Gamma\colon [0,+\infty) \to \H$ \emph{lands at infinity}; i.e.
\[
\lim_{t\to+\infty}\Gamma(t)=\infty.
\]
In this setting we have that if
\[
\lim_{t\to+\infty}\mathrm{arg}(1-\overline{\tau}\gamma(t))=\theta\in\left[-\tfrac{\pi}{2},\tfrac{\pi}{2}\right],
\]
then
\[
\lim_{t\to+\infty}\mathrm{arg}(\Gamma(t))=\theta.
\]
So, examining the angle with which curves land at the boundary is a far simpler endeavour in $\H$.

\medskip

Since all the quantities considered in our main result are conformally invariant, we can restate Theorem \ref{thm: main result} in the setting of right half-plane as follows:

\begin{theorem}\label{thm: main monotonicity result}
    Let $f\colon\H\to\H$ be a hyperbolic map with Denjoy--Wolff point $\infty$, and $\gamma\colon [0,+\infty)\to\H$ a curve landing at $\infty$, satisfying
    \[
    \lim_{t\to+\infty}\mathrm{arg}(\gamma(t))=\theta\in\left(-\tfrac{\pi}{2},\tfrac{\pi}{2}\right).
    \]
    Fix $z\in\D$ and let $\pi_\gamma(f^n(z))$ be a projection of $f^n(z)$ onto $\gamma$, for $n\in\N$. Then, for any $w\in\H$, the sequence $\{d_\H(w,\pi_\gamma(f^n(z)))\}$ is eventually strictly increasing.
\end{theorem}

Similarly, the right half-plane version of Corollary \ref{cor: main cor} is:

\begin{cor}\label{cor: main cor H version}
     Let $f\colon\H\to\H$ be a hyperbolic map with Denjoy--Wolff point $\infty$. For any $z,w\in\H$, the sequence $\{d_\H(w,f^n(z))\}$ is eventually strictly increasing.
\end{cor}

We end this subsection with an important property regarding the hyperbolic step and the slope of convergence of hyperbolic maps. For the proofs, we refer to \cite[Theorem 4.3.4, Lemma 4.6.2, Corollary 4.6.9]{Abate}.

\begin{theorem}\label{thm: step and slope for hyperbolic maps}
    If $f\colon\H\to\H$ is a hyperbolic map, then for every $z\in\H$ there exist constants $\phi\in(-\tfrac{\pi}{2},\tfrac{\pi}{2})$ and $d\in(0,+\infty)$ so that
    \begin{enumerate}[label=\textnormal{(\arabic*)}]
        \item $\displaystyle \lim_{n\to+\infty}d_\H(f^{n+1}(z),f^n(z))=d$; and
        \item $\displaystyle \lim_{n\to+\infty} \mathrm{arg}(f^n(z))=\phi$.
    \end{enumerate}
\end{theorem}

\subsection{Key lemmas}
In this final preliminary subsection we prove several results involving the hyperbolic geometry of $\H$ that are crucial to our analysis. The first is a restatement of \cite[Lemma 4.3]{BCDG}. A proof can be obtained easily from direct calculations on the formula \eqref{eq:hyperbolic distance in H}. 

\begin{lm}[\cite{BCDG}]\label{lm: properties of metric in H}
    Let $\theta \in(-\tfrac{\pi}{2},\tfrac{\pi}{2})$. 
    \begin{enumerate}[label=\textnormal{(\arabic*)}]
        \item For every $r_0>0$ and $\phi \in(-\tfrac{\pi}{2},\tfrac{\pi}{2})$ the function $r\mapsto d_\H(re^{i\phi},r_0e^{i\theta})$, $r\in(0,+\infty)$ has a minimum for $r=r_0$, is strictly decreasing for $r<r_0$ and strictly increasing for $r>r_0$.
        \item For any $r>0$ and $\phi\in(-\tfrac{\pi}{2},\tfrac{\pi}{2})$, we have $d_\H(re^{i\theta},re^{i\phi})=d_\H(e^{i\theta},e^{i\phi})$. Also, the function $(-\tfrac{\pi}{2},\tfrac{\pi}{2})\ni\phi\mapsto d_\H(e^{i\theta},e^{i\phi})$ is strictly decreasing in $(-\tfrac{\pi}{2},\theta)$ and strictly increasing in $(\theta, \tfrac{\pi}{2})$.
    \end{enumerate}
\end{lm}

Using Lemma \ref{lm: properties of metric in H} we can prove the following minor generalization of \cite[Lemma 4.4]{BCDG}, whose proof is almost identical to \cite[Lemma 4.4]{BCDG} and is thus omitted. 

\begin{lm}[\cite{BCDG}]\label{lm: sectors in H}
    Let $\gamma\colon [0,+\infty)\to\H$ be the straight half-line with $\gamma([0,+\infty))=\{re^{i\theta}\colon\ r\geq r_0\}$, for some $\theta\in(-\tfrac{\pi}{2}, \tfrac{\pi}{2}) $ and $r_0>0$. For every $R>0$, there exist $\phi_1\in(-\tfrac{\pi}{2},\theta)$ and $\phi_2\in(\theta, \tfrac{\pi}{2})$ satisfying $d_\H(e^{i\theta},e^{i\phi_1})=d_\H(e^{i\theta},e^{i\phi_2})=R$, so that 
    \[
    \{z\in\H\colon\ d_\H(z,\gamma)<R\}=D_\H(r_0e^{i\theta},R)\cup\{re^{i\phi} \colon r> r_0 \ \text{and} \ \phi\in(\phi_1,\phi_2)\}.
    \]
\end{lm}

\medskip

\begin{lm}\label{lm:sequences lemma in H}
    Let $\{r_n\},\{R_n\}\subset (0,+\infty)$ and $\{\theta_n\},\{\phi_n\}\subset(-\tfrac{\pi}{2},\tfrac{\pi}{2})$ be sequences so that both $\{\theta_n\}$ and $\{\phi_n\}$ converge to some $\theta\in(-\tfrac{\pi}{2},\tfrac{\pi}{2})$. The limit 
    \[
    \lim_{n\to+\infty}d_\H(r_n,R_n)
    \]
    exists if and only if the limit
    \[
    \lim_{n\to+\infty}d_\H(r_ne^{i\theta_n}, R_ne^{i\phi_n})
    \]
    exists. Moreover, if any of the above limits is zero or $+\infty$, so is the other. 
\end{lm}

\begin{proof} From the formula of the hyperbolic metric in $\H$ \eqref{eq:hyperbolic distance in H}, we get that 
    \begin{equation}\label{eq:half-plane lemma, eq1}
	 d_\H(r_n,R_n)=\frac{1}{2}\log\frac{1+\left\lvert\frac{R_n-r_n}{R_n+r_n}\right\rvert}{1-\left\lvert\frac{R_n-r_n}{R_n+r_n}\right\rvert}\\
	=\frac{1}{2}\log\frac{1+\left\lvert\frac{\frac{R_n}{r_n}-1}{\frac{R_n}{r_n}+1}\right\rvert}{1-\left\lvert\frac{\frac{R_n}{r_n}-1}{\frac{R_n}{r_n}+1}\right\rvert}.   
	\end{equation}
    and 
    \begin{equation}\label{eq:half-plane lemma, eq2}
        d_\H(r_ne^{i\theta_n}, R_ne^{i\phi_n})=\frac{1}{2}\log\frac{1+\left\lvert\frac{R_ne^{i\phi_n}-r_ne^{i\theta_n}}{R_ne^{i\phi_n}+r_ne^{-i\theta_n}}\right\rvert}{1-\left\lvert\frac{R_ne^{i\phi_n}-r_ne^{i\theta_n}}{R_ne^{i\phi_n}+r_ne^{-i\theta_n}}\right\rvert}.
    \end{equation}
    Simple calculations show that 
    \begin{equation}\label{eq:half-plane lemma, eq3}
    \left\lvert\frac{R_ne^{i\phi_n}-r_ne^{i\theta_n}}{R_ne^{i\phi_n}+r_ne^{-i\theta_n}}\right\rvert^2 =\frac{\left(\frac{R_n}{r_n}\right)^2-2\frac{R_n}{r_n}\cos(\theta_n-\phi_n)+1}{\left(\frac{R_n}{r_n}\right)^2+2\frac{R_n}{r_n}\cos(\theta_n+\phi_n)+1}.
    \end{equation}
    From these last three equations, \eqref{eq:half-plane lemma, eq1}, \eqref{eq:half-plane lemma, eq2} and \eqref{eq:half-plane lemma, eq3}, and the fact that $\theta_n$ and $\phi_n$ converge to $\theta\in(-\tfrac{\pi}{2},\tfrac{\pi}{2})$, we see that
    \[
    \lim_{n\to+\infty}d_\H(r_n,R_n)=0\iff \lim_{n\to+\infty}\frac{R_n}{r_n}=1\iff \lim_{n\to+\infty}d_\H(r_ne^{i\theta_n}, R_ne^{i\phi_n})=0.
    \]
    Moreover, $\lim\limits_{n\to+\infty}d_\H(r_n,R_n)=+\infty$ if and only if 
    \begin{equation}\label{eq:half-plane lemma, eq+}
    \text{either} \ \lim_{n\to+\infty}\frac{R_n}{r_n}=0, \ \text{or}\ \lim_{n\to+\infty}\frac{R_n}{r_n}=+\infty.
    \end{equation}
    But, due to \eqref{eq:half-plane lemma, eq2} and \eqref{eq:half-plane lemma, eq3}, the condition \eqref{eq:half-plane lemma, eq+} is equivalent to 
    \[
    \lim\limits_{n\to+\infty} d_\H(r_ne^{i\theta_n}, R_ne^{i\phi_n})=+\infty.
    \]
    Hence, we obtain
    \[
        \lim_{n\to+\infty}d_\H(r_n,R_n)=+\infty\iff  \lim_{n\to+\infty}d_\H(r_ne^{i\theta_n}, R_ne^{i\phi_n})=+\infty.
    \]
    So, for the rest of the proof, we can assume that neither of the sequences $\{d_\H(R_n,r_n)\}$ and $\{d_\H(r_ne^{i\theta_n}, R_ne^{i\phi_n})\}$ accumulate to $0$ or $+\infty$. This means that $\{\frac{R_n}{r_n}\}$ does not accumulate to 0, 1 or $+\infty$. We first assume that 
    \begin{equation}\label{eq:half-plane lemma, eq++}
        \lim_{n\to+\infty}d_\H(R_ne^{i\phi_n},r_ne^{i\theta_n})\in(0,+\infty).
    \end{equation}
    Suppose that $x,y$ are two distinct accumulation points of the sequence $\{\frac{R_n}{r_n}\}$. According to our assumptions, $x,y\in(0,1)\cup(1,+\infty)$. Now, \eqref{eq:half-plane lemma, eq2} and the convergence in \eqref{eq:half-plane lemma, eq++} implies that the sequence
    \[
    \left\{\left\lvert\frac{R_ne^{i\phi_n}-r_ne^{i\theta_n}}{R_ne^{i\phi_n}+r_ne^{-i\theta_n}}\right\rvert\right\}
    \]
    converges. So, \eqref{eq:half-plane lemma, eq3} implies that
    \begin{equation}\label{eq:half=plane lemma, eq4}
    \frac{x^2-2x+1}{x^2+2x\cos(2\theta)+1}=\frac{y^2-2y+1}{y^2+2y\cos(2\theta)+1}.
    \end{equation}
    Consider the real function $g(r)=\frac{r^2-2r+1}{r^2+2r\cos(2\theta)+1}$, for $r>0$. Observe that $g(r)=g(\frac{1}{r})$, for all $r>0$. Also, $g$ is strictly decreasing in $(0,1)$ and strictly increasing in $(1,+\infty)$. Therefore, we can easily see that that $g(r_1)=g(r_2)$ for some $r_1,r_2>0$ if and only if either $r_1=r_2$ or $r_1r_2=1$. Since \eqref{eq:half=plane lemma, eq4} is exactly $g(x)=g(y)$, we get that either $x=y$ or $xy=1$. In any case
    \[
    \left|\frac{x-1}{x+1}\right|=\left|\frac{y-1}{y+1}\right|,
    \]
    which by \eqref{eq:half-plane lemma, eq1} yields the convergence of $\{d_\H(r_n,R_n)\}$. Conversely, if 
    \[
    \lim_{n\to+\infty} d_\H(r_n,R_n)\in (0,+\infty), 
    \]
    then by \eqref{eq:half-plane lemma, eq1} the sequence 
    \[
    \left\{ \left\lvert\frac{\frac{R_n}{r_n}-1}{\frac{R_n}{r_n}+1}\right\rvert\right\}
    \]
    converges. Just like before, we assume that $x,y\in(0,1)\cup(1,+\infty)$ are two accumulation points of the sequence $\{\frac{R_n}{r_n}\}$. Then, we have that
    \[
    \left|\frac{x-1}{x+1}\right|=\left|\frac{y-1}{y+1}\right|
    \]
    which implies that either $x=y$ or $xy=1$, and then we can apply the arguments above in the reverse order to obtain the convergence of $\{d_\H(R_ne^{i\phi_n},r_ne^{i\theta_n})\}$.
\end{proof}

\medskip

\begin{lm}\label{lm: differences lemma}
     Let $\{r_n\}\subset \R^+$ be a sequence converging to $+\infty$ and $\theta\in(-\tfrac{\pi}{2},\tfrac{\pi}{2})$. For any $r_0>0$, we have that
     \[
     \lim_{n\to+\infty}\left( d_\H(r_0e^{i\theta},r_ne^{i\theta}) - d_\H(r_0,r_n)\right) = -\log \left(\cos\theta\right) \in[0,+\infty).
     \]
\end{lm}

\begin{proof}
    Using the formulas \eqref{eq:hyperbolic distance in H} and \eqref{eq: pseudo-hyperbolic metric}, for $d_\H$ and $\rho_\H$, we obtain that for all $n\in\N$
    \[
    d_\H(r_0e^{i\theta},r_ne^{i\theta}) - d_\H(r_0,r_n) = \log\frac{1+\rho_\H(r_0e^{i\theta},r_ne^{i\theta})}{1+\rho_\H(r_0,r_n)} + \frac{1}{2}\log\frac{1-\rho_\H(r_0,r_n)^2}{1-\rho_\H(r_0e^{i\theta},r_ne^{i\theta})^2}.
    \]
    Since $r_n$ converges to $+\infty$ and $r_0$ is fixed, we have that the pseudo-hyperbolic distances $\rho_\H(r_0,r_n)$ and $\rho_\H(r_0e^{i\theta},r_ne^{i\theta})$ both converge to 1. Therefore,
    \[
    \lim_{n\to+\infty}\log\frac{1+\rho_\H(r_0e^{i\theta},r_ne^{i\theta})}{1+\rho_\H(r_0,r_n)} =0.
    \]
    So it suffices to show that 
    \[
    \lim_{n\to+\infty}\frac{1-\rho_\H(r_0,r_n)^2}{1-\rho_\H(r_0e^{i\theta},r_ne^{i\theta})^2} =\frac{1}{\cos^2\theta}.
    \]
    This easily follows from the next string of equations, obtained by using the pseudo-hyperbolic formula \eqref{eq: useful pseudo-hyp formula}
    \begin{align*}
        \frac{1-\rho_\H(r_0,r_n)^2}{1-\rho_\H(r_0e^{i\theta},r_ne^{i\theta})^2}&=\frac{\lvert r_0e^{i\theta}+r_ne^{-i\theta}\rvert^2}{\cos^2\theta\ (r_0+r_n)^2} = \frac{r_0^2+2r_0r_n\cos2\theta+r_n^2}{\cos^2\theta\ (r_0+r_n)^2}\\
        &=\frac{\left(\frac{r_0}{r_n}\right)^2+2\frac{r_0}{r_n}\cos(2\theta)+1}{\left(\frac{r_0}{r_n}+1\right)^2\cos^2\theta}.\qedhere
    \end{align*}
\end{proof}

\section{Examples}\label{sect: examples}

Here we present the examples which show that Theorem \ref{thm: main result} is best possible, in the sense described in the Introduction. All our examples will be constructed in the right half-plane, and so we refer to Theorem \ref{thm: main monotonicity result} instead, which is the equivalent formulation of Theorem \ref{thm: main result} in $\H$. 

\begin{ex}\label{ex: hyperbolic semigroup}
    We first show that the strong monotonicity of Theorem \ref{thm: main monotonicity result} fails for continuous semigroups of holomorphic self-maps of $\H$. Let us first properly define such a semigroup and discuss some of its basic properties. For a complete reference on continuous semigroups of holomorphic functions, see \cite[Chapter 8]{BCDM-Book}.
    
    \medskip
    
    A \emph{semigroup $(\phi_t)$ in $\H$} is a family of holomorphic functions $\phi_t\colon \H\to \H$, for $t\geq0$, satisfying the following properties:
    \begin{enumerate}[label=\textnormal{(\roman*)}]
        \item $\phi_0=\mathrm{id}_\H$;
        \item $\phi_{t+s}=\phi_t\circ\phi_s$, for all $t,s\ge0$; and
        \item $\lim_{t\to0}\phi_t(z)=z$, for all $z\in\H$.
    \end{enumerate}
    Notice that conditions (ii) and (iii) imply that $\phi_t$ depends continuously on the real parameter $t$, in the topology of uniform convergence on compact sets.
    
    The continuous version of the Denjoy--Wolff theorem \cite[Theorem 8.3.1]{BCDM-Book} states that if the function $\phi_{t_0}$, for some $t_0>0$, does not have any fixed points in $\H$, then there exists a point $\tau\in\partial\H$ such that $\phi_t$ converges to the constant $\tau$, as $t\to+\infty$, uniformly on compact subsets of $\H$. Here, by $\partial\H$ we mean the imaginary axis along with the point infinity. The point $\tau$ is called the Denjoy--Wolff point of the semigroup $(\phi_t)$. Moreover, the angular derivative
    \begin{equation}\label{eq:angular der for semigroups}
    \phi'_t(\tau)=\angle\lim_{z\to\tau}\frac{\phi_t(z)}{z},
    \end{equation}
    exists for all $t$ and satisfies $\phi'_t(\tau)\in[1,+\infty)$. Just like for self-maps of $\H$, the semigroup $(\phi_t)$ is called \emph{hyperbolic} if $\phi'_t(\tau)>1$, for some (and hence any) $t>0$.
    
    Returning to our example, consider the functions $\phi_t(z)=e^t z$, $z\in\H$, for $t\ge0$. The family $(\phi_t)$ is clearly a hyperbolic semigroup in $\H$, with Denjoy--Wolff point $\infty$.

    \medskip
    
    Let $\delta:[0,+\infty)\to\H$ be a curve with trace $\delta([0,+\infty))=H\cup\left(\bigcup_{n\in\N}C_n\right)$, where $H$ is the horizontal half-line $\{t+i:t\ge0\}$ and $C_n$ is the hyperbolic geodesic segment of $\H$ joining $e^n$ with $\sqrt{e^{2n}-1}+i\in H$; see Figure \ref{fig: hyperbolic semigroup}. 
    
    We can parametrise $\delta$ so that $\lim_{t\to+\infty}\delta(t)=\infty$. Then, since the trace of the curve is bounded between two horizontal half-lines, we can easily see that $\lim_{t\to+\infty}\arg(\delta(t))=0$. As a result, all assumptions of Theorem \ref{thm: main monotonicity result} hold, but with a hyperbolic semigroup instead of a hyperbolic function.
    
    However, because the points $e^n$ lie in $\delta([0,+\infty))$ by construction, we have that $\pi_\delta(e^n)=e^n$. Moreover, the real axis intersects each $C_n$ orthogonally, meaning that the projection of any point $x>0$ onto the geodesic $C_n$ is $e^n$, for all $n\in\N$. So, for each $n\in\N$ there exists some $\epsilon_n>0$ small enough, so that $\pi_\delta(e^t)=e^n$, for all $t\in(n-\epsilon_n,n+\epsilon_n)$. In other words, the distance $d_\H(1,\pi_\delta(\phi_t(1)))=d_\H(1,\pi_\delta(e^t))$ remains constant on each interval $(n-\epsilon_n,n+\epsilon_n)$, $n\in\N$, which shows that $(\phi_t)$ does not satisfy the monotonicity result described in our main result.
\end{ex}

\begin{figure}[ht]
    \centering
    \begin{tikzpicture}[scale=2.1]
    \draw[->] (.9,-0.3) -- (.9,1.5);
            \draw[ dotted] (.9,0) -- (5.2,0);
        \draw plot[domain=0.:0.7297276562269663,variable=\t]({1.*1.5*cos(\t r)+0.*1.5*sin(\t r)},{0.*1.5*cos(\t r)+1.*1.5*sin(\t r)});
\draw  plot[domain=0.:0.46055399168132244,variable=\t]({1.*2.25*cos(\t r)+0.*2.25*sin(\t r)},{0.*2.25*cos(\t r)+1.*2.25*sin(\t r)});
\draw   plot[domain=0.:0.30081247525662086,variable=\t]({1.*3.375*cos(\t r)+0.*3.375*sin(\t r)},{0.*3.375*cos(\t r)+1.*3.375*sin(\t r)});
\draw  (1,1.)-- (5.2,1.);
\draw  plot[domain=0.:0.19883851477848238,variable=\t]({1.*5.0625*cos(\t r)+0.*5.0625*sin(\t r)},{0.*5.0625*cos(\t r)+1.*5.0625*sin(\t r)});
\filldraw[black] (1.5,0) circle (.5pt) node[below] {$e$};
\filldraw[black] (1.12,1) circle (.5pt) node[above right, xshift=-5pt, rotate=25] {$\sqrt{e^2-1}+i$};
\filldraw[black] (2.25,0) circle (.5pt) node[below] {$e^2$};
\filldraw[black] (2.02,1) circle (.5pt) node[above right, xshift=-5pt, rotate=25] {$\sqrt{e^4-1}+i$};
\filldraw[black] (3.38,0) circle (.5pt) node[below] {$e^3$};
\filldraw[black] (3.22,1) circle (.5pt) node[above right, xshift=-5pt, rotate=25] {$\sqrt{e^6-1}+i$};
\filldraw[black] (5.06,0) circle (.5pt) node[below] {$e^4$};
\filldraw[black] (4.96,1) circle (.5pt) node[above right, xshift=-5pt, rotate=25] {$\sqrt{e^8-1}+i$};
\filldraw[black] (1,1) circle (.5pt) node[yshift=-15pt, rotate=90] {$1+i$};
    \end{tikzpicture}
    \caption{Example for a hyperbolic semigroup.}\label{fig: hyperbolic semigroup}
\end{figure}
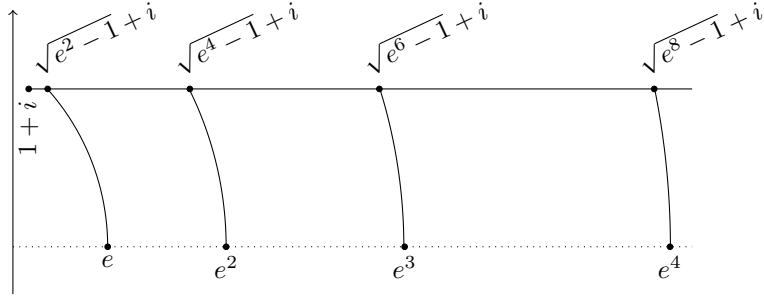

\begin{ex}\label{ex: parabolic map}
    In this second example, we are going to verify that Theorem \ref{thm: main monotonicity result} fails in general for parabolic self-maps of the right half-plane. We consider two examples, one for each main type of parabolic maps. It is known (see, for instance, \cite[Section 4.6]{Abate}) that for a parabolic map $f\colon \H\to \H$, we have that the limit $\lim_{n\to+\infty}d_\H(f^n(1),f^{n+1}(1))$ exists and is either zero or positive. The former type of parabolic map is called \emph{parabolic of zero hyperbolic step}, while the latter \emph{parabolic of positive hyperbolic step}.
    
    First, consider the holomorphic function $f\colon\H\to\H$, with $f(z)=z+1$. It is clear that $f'(\infty)=1$, and thus $f$ is parabolic. Also, $d_\H(f^{n+1}(1),f^n(1))=\tfrac{1}{2}\log\tfrac{n+2}{n+1}$, due to \eqref{eq:hyperbolic distance in H}, meaning that $f$ is parabolic of zero hyperbolic step.\\
    Let $\delta:[0,+\infty)\to\H$ be a curve with trace $\delta([0,+\infty))=H\cup\left(\bigcup_{n\in\N}C_n\right)$, where $H$ is the horizontal half-line $\{t+i:t\ge1\}$ and $C_n$ is the hyperbolic geodesic segment of $\H$ joining $3n$ with $\sqrt{(3n)^2-1}+i$; see Figure \ref{fig: parabolic map}. Similarly to the previous example, we may write $\lim_{t\to+\infty}\delta(t)=\infty$ and $\lim_{t\to+\infty}\arg(\delta(t))=0$. Thus, $\delta$ satisfies the assumptions of Theorem \ref{thm: main monotonicity result}.

    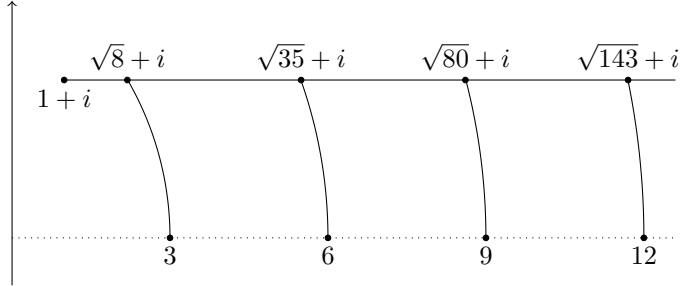
\begin{figure}[ht]
    \centering
    \begin{tikzpicture}[scale=2.1]
    \draw[->] (1,-0.3) -- (1,1.5);
            \draw[ dotted] (1,0) -- (5.2,0);
\draw plot[domain=0.:0.5235987755982989,variable=\t]({1.*2.*cos(\t r)+0.*2.*sin(\t r)},{0.*2.*cos(\t r)+1.*2.*sin(\t r)});
\draw plot[domain=0.:0.3398369094541219,variable=\t]({1.*3.*cos(\t r)+0.*3.*sin(\t r)},{0.*3.*cos(\t r)+1.*3.*sin(\t r)});
\draw plot[domain=0.:0.25268025514207865,variable=\t]({1.*4.*cos(\t r)+0.*4.*sin(\t r)},{0.*4.*cos(\t r)+1.*4.*sin(\t r)});
\draw plot[domain=0.:0.2013579207903308,variable=\t]({1.*5.*cos(\t r)+0.*5.*sin(\t r)},{0.*5.*cos(\t r)+1.*5.*sin(\t r)});
\draw  (1.33,1.)-- (5.2,1.);

\filldraw[black] (1.33,1) circle (.5pt) node[below] {$1+i$};
\filldraw[black] (2,0) circle (.5pt) node[below] {$3$};
\filldraw[black] (1.73,1) circle (.5pt)  node[above] {$\sqrt{8}+i$};
\filldraw[black] (3,0) circle (.5pt) node[below] {$6$};
\filldraw[black] (2.83,1) circle (.5pt)  node[above] {$\sqrt{35}+i$};
\filldraw[black] (4,0) circle (.5pt) node[below] {$9$};
\filldraw[black] (3.87,1) circle (.5pt)  node[above] {$\sqrt{80}+i$};
\filldraw[black] (5,0) circle (.5pt) node[below] {$12$};
\filldraw[black] (4.9,1) circle (.5pt)  node[above] {$\sqrt{143}+i$};
 \end{tikzpicture}
    \caption{Example for a parabolic map with zero hyperbolic step.}
    \label{fig: parabolic map}
\end{figure}

    Now, consider the points $f^{3n-2}(1)=3n-1$, for $n\in\N$, which form a subsequence of the orbit $\{f^n(1)\}$. By construction, the projection of $3n-1$ onto $\delta$ is either its projection onto the half-line $H$, or its projection onto one of the circular arcs $C_n$. Note that $H$ and all the $C_n$ are hyperbolic geodesic segments of $\H$, and thus the aforementioned projections are unique. Given that the real axis is orthogonal to every $C_k$, we can see that the projection of $3n-1$ onto each $C_k$ is the point $3k$. Therefore, the point of the set $\bigcup_{k\in\N}C_k$ that is closest, in hyperbolic terms, to $3n-1$ is either $3n$ or $3n-3$. Using \eqref{eq:hyperbolic distance in H}, we easily check that 
    \[
    d_\H(3n-3,3n-1)>d_\H(3n-1,3n).
    \]
    So, the point of the set $\bigcup_{k\in\N}C_k$ closest to $3n-1$ is indeed $3n$. We claim that $\pi_\delta(3n-1)=3n$. Note that for any $x>0$, the Euclidean circle centred at $i$ and of radius $\lvert x-i\rvert$ is a hyperbolic geodesic of $\H$ perpendicular to the half-line $H$. So, the projection of $x$ onto $H$ is the point of intersection of the aforementioned circle and $H$. Simple calculations show that this point is $\sqrt{x^2+1}+i$. Consequently, the projection of $3n-1$ onto the half-line $H$ is $\sqrt{(3n-1)^2+1}+i$. Therefore, our goal is to show that 
    \begin{equation}\label{eq:ex2}
        d_\H\left(\sqrt{(3n-1)^2+1}+i, 3n-1\right)>d_\H(3n, 3n-1). 
    \end{equation}
    Due to the connection between the hyperbolic metric $d_\H$ and the pseudo-hyperbolic metric $\rho_\H$ evident in the formula \eqref{eq:hyperbolic distance in H}, it suffices to show that
    \begin{equation}\label{eq:ex3}
        \rho_\H\left(\sqrt{(3n-1)^2+1}+i, 3n-1\right)>\rho_\H(3n, 3n-1).
    \end{equation}
    Now, simple computations using the formula \eqref{eq: pseudo-hyperbolic metric} of $\rho_\H$ yield
    \begin{eqnarray*}
        \rho_\H(3n-1,\sqrt{9n^2+1}+i)&=&\dfrac{1}{\sqrt{(3n-1)^2+1}+3n-1}\\
        &>&\frac{1}{2(3n-1)+1}=\rho_\H(3n-1,3n),
    \end{eqnarray*}
    which is exactly the desired inequality \eqref{eq:ex3}.
    
    So far we have shown that $\pi_\delta(f^{3n-2}(1))=3n$, for all $n\in\N$. Moreover, since, by construction, we have that $3n\in\delta([0,+\infty))$, for all $n\in\N$, we immediately get that $\pi_\delta(f^{3n-1}(1))=\pi_\delta(3n)=3n$.  To sum up, we have that
    \[
    \pi_\delta(f^{3n-2}(1))=\pi_\delta(f^{3n-1}(1)), \quad \text{for all } n\in\N,
    \]
    and thus the sequence $\{d_{\H}(1,\pi_\delta(f^k(1)))\}$ is not eventually strictly increasing.

    The example for parabolic maps of positive hyperbolic step is similar, so we provide a rough sketch of the construction. Let $g\colon \H\to \H$ be the map with $g(z)=z+i$. It is easy to see that $g$ is a parabolic automorphism of $\H$, and thus has positive hyperbolic step (see \cite[Lemma 4.6.3]{Abate}). Consider the curve $\gamma\colon[0,+\infty)\to\H$, with trace $\gamma([0,+\infty))=H'\cup\left(\bigcup_{n\in\N}C_n'\right)$, where $H'$ is the horizontal half-line $\{t-i\colon t\geq 3\}$ and $C'_n$ is the hyperbolic geodesic joining $\sqrt{1+(2n+1)^2}$ and $(2n+1)-i$, as in Figure \ref{fig: parabolic map 2}.
    \begin{figure}[ht]
    \centering
    \begin{tikzpicture}[scale=2.1]
    \draw[->] (1,-1.3) -- (1,0.3);
            \draw[ dotted] (1,0) -- (5.2,0);
\draw plot[domain=5.75958653158129:6.28318530717959,variable=\t]({1.*2.*cos(\t r)+0.*2.*sin(\t r)},{0.*2.*cos(\t r)+1.*2.*sin(\t r)});
\draw plot[domain=5.94334839772546:6.28318530717959,variable=\t]({1.*3.*cos(\t r)+0.*3.*sin(\t r)},{0.*3.*cos(\t r)+1.*3.*sin(\t r)});
\draw plot[domain=6.03050505203751:6.28318530717959,variable=\t]({1.*4.*cos(\t r)+0.*4.*sin(\t r)},{0.*4.*cos(\t r)+1.*4.*sin(\t r)});
\draw plot[domain=6.08182738638926:6.28318530717959,variable=\t]({1.*5.*cos(\t r)+0.*5.*sin(\t r)},{0.*5.*cos(\t r)+1.*5.*sin(\t r)});

\draw  (1.73,-1.)-- (5.2,-1.);

\filldraw[black] (1.73,-1) circle (.5pt) node[below] {$3-i$};
\filldraw[black] (2,0) circle (.5pt) node[above] {$\sqrt{10}$};
\filldraw[black] (3,0) circle (.5pt) node[above] {$\sqrt{26}$};
\filldraw[black] (2.83,-1) circle (.5pt)  node[below] {$5-i$};
\filldraw[black] (4,0) circle (.5pt) node[above] {$\sqrt{50}$};
\filldraw[black] (3.87,-1) circle (.5pt)  node[below] {$7-i$};
\filldraw[black] (5,0) circle (.5pt) node[above] {$\sqrt{82}$};
\filldraw[black] (4.9,-1) circle (.5pt)  node[below] {$9-i$};
 \end{tikzpicture}
    \caption{Example for a parabolic map with positive hyperbolic step.}
    \label{fig: parabolic map 2}
\end{figure}
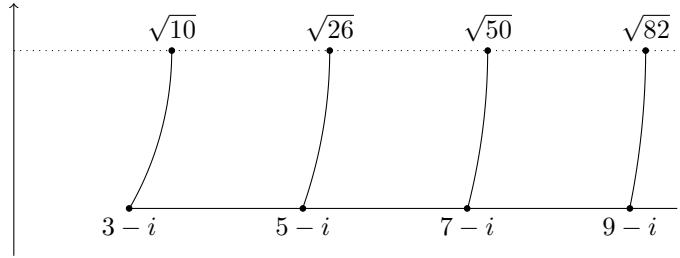
    Note that $\lvert g^n(1)\rvert=\lvert 1+ni\rvert=\sqrt{1+n^2}$, and so following arguments similar to those presented in the example for the zero hyperbolic step we have that the projection of $g^{2n+1}(1)$ onto $\gamma$ is the point $\sqrt{1+(2n+1)^2}$.\\
    We now need to evaluate the projection of $g^{2n+2}(1)$ onto $\gamma$. Firstly, the projection of $g^{2n+2}(1)$ onto the half-line $H'$ is $\sqrt{1+(2n+2)^2}-i$. Also, observe that the segment $C_n'$ lies in the geodesic $L_n=\left\{z\in\H\colon \lvert z \rvert =\sqrt{1+(2n+1)^2}\right\}$. Simple Euclidean arguments show that the hyperbolic projection of the point $g^{2n+2}(1)= 1+(2n+2)i$ onto $L_n$ lies in the upper half-plane. This means that the point of $C_n'$ closest (in hyperbolic terms) to $g^{2n+2}(1)$ is $\sqrt{1+(2n+1)^2}$. Similarly the point of $C_{n+1}'$ closest to $g^{2n+2}(1)$ is $\sqrt{1+(2n+3)^2}$. We therefore have that the projection of $g^{2n+2}(1)$ onto $\gamma$ is one of the points $\sqrt{1+(2n+2)^2}-i$, $\sqrt{1+(2n+1)^2}$ or $\sqrt{1+(2n+3)^2}$. By performing long, but elementary computations using the pseudo-hyperbolic formula \eqref{eq: useful pseudo-hyp formula}, one can verify that 
    \[
    \rho_\H\left(g^{2n+2}(1), \sqrt{1+(2n+1)^2}\right) < \rho_\H\left(g^{2n+2}(1), \sqrt{1+(2n+2)^2}-i\right),
    \]
    and
    \[
    \rho_\H\left(g^{2n+2}(1), \sqrt{1+(2n+1)^2}\right) < \rho_\H\left(g^{2n+2}(1), \sqrt{1+(2n+3)^2}\right).
    \]
    We conclude that $\pi_\gamma(g^{2n+2}(1))=\sqrt{1+(2n+1)^2} = \pi_\gamma(g^{2n+1}(1))$, for all $n\in\N$, implying that the sequence $\{d_\H(1,\pi_\gamma(g^n(1)))\}$ is not eventually strictly increasing.    
\end{ex}

\begin{ex}\label{ex: hyperbolic map}
    Here, we present an example demonstrating that the ``eventually" part of the monotonicity in Theorem \ref{thm: main monotonicity result} cannot be omitted. Let $f\colon\H\to\H$ be the hyperbolic function with $f(z)=2z$.

    Consider the curve $\delta\colon [0,+\infty)\to\H$, with trace 
    \[
    \delta([0,+\infty))=\partial D_\H\left(2, \log\sqrt[4]{2}\right)\cup \partial D_\H\left(4, \log\sqrt[4]{2}\right) \cup\left\{t\in\R\colon t\geq 4\sqrt{2}\right\},
    \]
    where $\partial D_\H(z,r)$ is the hyperbolic circle centred at $z\in\H$ and of radius $r>0$ (see Figure \ref{fig: hyperbolic map}). 

    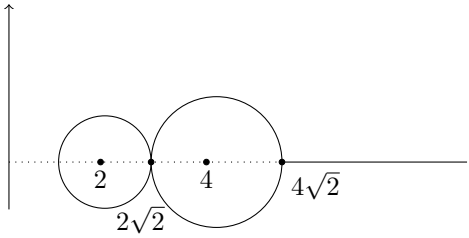
\begin{figure}[ht]
    \centering
    \begin{tikzpicture}[scale=2.1]
    \draw[->] (1.1,-0.3) -- (1.1,1);
            \draw[ dotted] (1.1,0) -- (2.828,0);
\draw (1.7071067811865475,0.) circle (0.2928932188134523cm);
\draw (2.414213562373095,0.) circle (0.4142135623730949cm);
\filldraw[black] (1.68,0) circle (.5pt) node[below] {$2$};
\filldraw[black] (2.35,0) circle (.5pt) node[below] {$4$};
\filldraw[black] (2.82842712474619,0) circle (.5pt) node[below right] {$4\sqrt{2}$};
\filldraw[black] (2,0) circle (.5pt) node[below, yshift=-13pt, xshift=-4pt] {$2\sqrt{2}$};
\draw (2.828,0) -- (4,0);
\end{tikzpicture}
    \caption{Example for a hyperbolic map.}\label{fig: hyperbolic map}
\end{figure}

The curve can be parametrised so that $\lim_{t\to+\infty}\delta(t)=\infty$, and so by construction there exists some $t_0>0$, so that $\arg(\delta(t))=0$, for all $t\geq t_0$. Therefore, $\delta$ satisfies the assumptions of Theorem \ref{thm: main monotonicity result}.

Using the formula \eqref{eq:hyperbolic distance in H} for $d_\H$, we can see that the circles $\partial D_\H\left(2, \log\sqrt[4]{2}\right)$ and $\partial D_\H\left(4, \log\sqrt[4]{2}\right)$ intersect at the point $2\sqrt{2}$. Moreover, we trivially have that all the points of $\partial D_\H\left(2, \log\sqrt[4]{2}\right)$ are equidistant from $2$ and so the projection of $f(1)=2$ onto $\delta$ can be chosen to be $2\sqrt{2}$. Similarly, the projection of $f^2(1)=4$ can be, again, chosen to be $2\sqrt{2}$, meaning that $\pi_\delta(f(1))=\pi_\delta(f^2(2))$. So, the sequence $d_\H(1,\pi_\delta(f^n(1))$ is strictly increasing only for $n\geq3$.
\end{ex}

\begin{ex}\label{ex: tangential slope example}
    Our next example shows that our main theorem fails if we assume that the curve $\gamma\colon[0,+\infty)\to\H$ satisfies
    \[
    \lim_{t\to+\infty}\mathrm{arg}(\gamma(t))=\pm\frac{\pi}{2}.
    \]
    For this, we let $\gamma$ be such that $\gamma(t)=1+it$, for all $t\geq0$. It is easy to see that $\arg(\gamma(t))$ converges to $\tfrac{\pi}{2}$. Now consider the hyperbolic map $f\colon\H\to\H$, with $f(z)=2z$, we used in Example \ref{ex: hyperbolic map}. Then, simple computations show that the projection of $f^n(1)=2^n$ onto $\gamma$ is the point $1$, for all $n\in\N$, and thus the sequence $\pi_\gamma(f^n(z))$ is constant.\\
    For the case where $\arg(\gamma(t))$ tends to $-\tfrac{\pi}{2}$ one can use the curve $\gamma(t)=1-it$.
\end{ex}

\begin{ex}\label{ex:corollary counterexample}
    Finally, we describe how the example constructed in \cite{BK} can be used to show that Corollary \ref{cor: main cor} may fail completely for parabolic functions. 

    In \cite[Section 4]{BK}, the authors construct a simply connected domain $\Omega\subsetneq\C$, with the following properties:
    \begin{enumerate}
        \item $z+t\in\Omega$, for all $t\geq0$ and any $z\in \Omega$;
        \item $\R\subset\Omega$;
        \item the family of holomorphic functions $\phi_t\colon \Omega\to\Omega$, with $\phi_t(z)=z+t$ is a parabolic semigroup in $\Omega$. 
    \end{enumerate}
    Let us remark that even though the definition of a semigroup presented in Example \ref{ex: hyperbolic semigroup} was stated only for the right half-plane, the definition can easily be extended to any simply connected domain via a Riemann map. Similarly, all the notation introduced for self-maps of $\H$ (or $\D$) can be easily transferred to $\Omega$. With this in mind, we will construct our example in $\Omega$. 

    A key point in their construction is that there exist real sequences $\{x_n\},\{y_n\}\subset \R$, tending to $+\infty$ such that $x_n<y_n<x_{n+1}$, for any $n\in\N$, and 
    \begin{equation}\label{eq:ex Bets-Kar 1}
         d_\Omega(0,x_n)>d_\Omega(0,y_n).
    \end{equation}
    In fact, $x_n$ and $y_n$ can be chosen to be positive integers (it is evident from \cite[Lemma 4.1]{BK} that one may chose $x_n=3^n-2^n$ and $y_n=3^n+2^n$).

    Now consider the function $\phi_1\colon\Omega\to\Omega$ from the semigroup $(\phi_t)$; that is, $\phi_1(z)=z+1$.  Since the semigroup $(\phi_t)$ is parabolic, every function $\phi_t$, for $t>0$, is also parabolic as a self-map of $\Omega$ (see the angular derivative condition \eqref{eq:angular der for semigroups} in Example \ref{ex: hyperbolic semigroup}). Note that $\phi_1^{x_n}(0)=x_n$ and $\phi_1^{y_n}(0)=y_n$. Thus, \eqref{eq:ex Bets-Kar 1} can be rewritten as
    \[
    d_\Omega(0,\phi_1^{x_n}(0))>d_\Omega(0,\phi_1^{y_n}(0)),
    \]
    meaning that the sequence $\{d_\Omega(0,\phi_1^n(0))\}$ is not eventually increasing. 
\end{ex}

\section{Projections of sequences onto curves}\label{sect:projections}

This section contains the bulk of our technical results about projections of arbitrary sequences onto curves. As usual, our proofs will be carried out in the right half-plane. 

Our first lemma requires some set-up. Let us denote by $\Delta_\H(z,r)$ the \emph{pseudo-hyperbolic disc} of $\H$, centred at $z\in\H$ and of radius $r\in(0,1)$; i.e.
\begin{equation}\label{eq: pseudo-hyp discs}
    \Delta_\H(z,r)=\{w\in\H\colon \rho_{\H}(w,z)<r\}.
\end{equation}

Using formula \eqref{eq: useful pseudo-hyp formula}, we can rewrite \eqref{eq: pseudo-hyp discs} as 
\begin{equation}\label{eq: pseudo-hyp discs v2}
    \Delta_\H(z,r)=\left\{w\in\H\colon \lvert \overline{w}+z\rvert^2<4\frac{\mathrm{Re}\ z}{1-r^2}\ \mathrm{Re}\ w\right\}
\end{equation}

The following lemma is a classical result \cite[Proposition 2.1.3]{Abate} quantifying the idea that horodiscs are limits of pseudo-hyperbolic discs, restated for the case of the right half-plane. We include a short proof for the sake of completeness.

\begin{lm}[\cite{Abate}]\label{lm: disc convergence lemma}
    Let $\{p_n\}\subset \H$ be a sequence of points converging to infinity and $\{r_n\}\subset(0,1)$ a sequence converging to 1. If
    \[
    \lim_{n\to+\infty}\frac{\mathrm{Re}\ p_n}{(1-r_n^2)\lvert 1+p_n\rvert^2}=\lambda\in (0,+\infty), 
    \]
    then 
    \[
    \left\{z\in\C\colon \mathrm{Re} \ z >\frac{1}{4\lambda} \right\}\subseteq \liminf_{n\to+\infty} \Delta_\H(p_n,r_n).
    \]
\end{lm}

\begin{proof}
    Arguing towards a contradiction, we assume that there exists a point $z_0\in\H$, with $\mathrm{Re}\ z_0>\tfrac{1}{4\lambda}$, and a sequence of integers $\{n_k\}$, so that $z_0\notin \Delta_\H(p_{n_k},r_{n_k})$, for all $k\in\N$. By \eqref{eq: pseudo-hyp discs v2}, this means that
    \begin{equation}\label{eq: disc convergence eq}
        4\ \mathrm{Re}\ z_0\frac{\mathrm{Re}\ p_{n_k}}{1-r_{n_k}^2}\leq \lvert \overline{z_0} +p_{n_k}\rvert^2 \leq \left(1+\frac{\lvert\overline{z_0} +1\rvert}{\lvert 1+p_{n_k}\rvert}\right)^2\ \lvert 1+p_{n_k}\rvert^2.
    \end{equation}
    Since $z_0$ is fixed and $p_n$ tends to infinity, we have that 
    \[
    \lim_{k\to+\infty}\frac{\lvert\overline{z_0} +1\rvert}{\lvert 1+p_{n_k}\rvert} =0.
    \]
    So, taking limits $k\to+\infty$ in \eqref{eq: disc convergence eq} yields 
    \[
    4\lambda \ \mathrm{Re}\ z_0\leq 1,
    \]
    which leads to a contradiction.
\end{proof}

We can now start our analysis by proving that in the setting of Theorem \ref{thm: main monotonicity result} (the half-plane version of Theorem \ref{thm: main result}), the projections tend to the point at infinity.

\begin{lm}\label{lm:convergence of projections to zeta}
    Let $\gamma\colon[0,+\infty) \to\H$ be a curve landing at infinity, with 
    \[
    \lim_{t\to+\infty}\mathrm{arg}(\gamma(t))=\theta, \quad  \text{for some}\ \theta\in(-\tfrac{\pi}{2}, \tfrac{\pi}{2}).
    \]
    If $\{z_n\}\subset\H$ is a sequence converging to infinity and $\pi_\gamma(z_n)$ is a projection of $z_n$ onto $\gamma$, for $n\in\N$, then
    \[
    \lim_{n\to+\infty}\pi_\gamma(z_n)=\infty.
    \]
\end{lm}

\begin{proof}
    For simplicity, we write $\pi_n\vcentcolon=\pi_\gamma(z_n)$. Since $\pi_n\in\gamma([0,+\infty))$, for all $n\in\N$, and $\gamma(t)$ converges to $\infty$, as $t\to+\infty$, it suffices to show that $\{\pi_n\}$ does not accumulate in $\H$. Assume, towards a contradiction, that the sequence $\{\pi_n\}$ is relatively bounded in $\H$. By passing to a subsequence if necessary, we can then assume that $\{\pi_n\}$ converges to some $z_0\in\H$, as $n\to+\infty$.\\
    Since $\pi_n$ is the projection of $z_n$ onto $\gamma$, the hyperbolic discs $D_\H(z_n,d_\H(z_n,\pi_n))$ do not intersect $\gamma([0,+\infty))$, for any $n\in\N$. Note that the formulas for the hyperbolic and pseudo-hyperbolic metrics, \eqref{eq:hyperbolic distance in H} and \eqref{eq: pseudo-hyperbolic metric} respectively, imply that $D_\H(z_n,d_\H(z_n,\pi_n))=\Delta_\H(z_n,\rho_\H(z_n,\pi_n))$, where $\Delta_\H$ are the pseudo-hyperbolic discs defined in \eqref{eq: pseudo-hyp discs}. Because $\{z_n\}$ converges to infinity, as $n\to+\infty$, we have that 
    \[
    \lim_{n\to+\infty}\rho_\H(z_n,\pi_n)=1.
    \]
    Our goal is to apply the Disc Convergence Lemma, Lemma \ref{lm: disc convergence lemma}, to the pseudo-hyperbolic discs $\Delta_\H(z_n,\rho(z_n,\pi_n))$. To do so, observe that the formula \eqref{eq: useful pseudo-hyp formula} implies that
    \begin{equation}\label{eq: convergence to zeta eq1}
        \frac{\mathrm{Re}\ z_n}{(1-\rho_\H(z_n,\pi_n)^2)\lvert 1+z_n\rvert^2}=\frac{\lvert \overline{z_n}+\pi_n\rvert^2}{4\mathrm{Re}\ \pi_n \lvert 1+z_n\rvert^2}.
    \end{equation}
    Because $\pi_n$ is bounded, we have that 
    \[
    \lim_{n\to+\infty}\frac{\lvert \overline{z_n}+\pi_n\rvert^2}{\lvert 1+z_n\rvert^2}=1.
    \]
    So, taking limits in \eqref{eq: convergence to zeta eq1} yields
    \[
    \lim_{n\to+\infty}\frac{\mathrm{Re}\ z_n}{(1-\rho_\H(z_n,\pi_n)^2)\lvert 1+z_n\rvert^2} = \frac{1}{4\ \mathrm{Re}\ z_0}>0.
    \]
    Thus, Lemma \ref{lm: disc convergence lemma} is indeed applicable to the sequence of discs $\Delta_\H(z_n,\rho_\H(z_n,\pi_n))$ and implies that 
    \begin{equation}\label{eq: convergence to zeta eq2}
        \{z\in\C\colon \mathrm{Re}\ z> \mathrm{Re}\ z_0\}\subseteq \liminf_{n\to+\infty} \Delta_\H(z_n,\rho_\H(z_n,\pi_n)).
    \end{equation}
    Recall that by assumption, $\mathrm{arg}(\gamma(t))$ converges to $\theta\in(-\tfrac{\pi}{2}, \tfrac{\pi}{2})$ and $\gamma(t)$ converges to $\infty$. This means that the set $\gamma([0,+\infty))$ intersects all half-planes $\{z\in\C\colon \mathrm{Re}\ z>x\}$, for $x>0$. Therefore, by \eqref{eq: convergence to zeta eq2} we get that 
    \[
    \gamma([0,+\infty))\cap\liminf_{n\to+\infty} \Delta_\H(z_n,\rho_\H(z_n,\pi_n))\neq\emptyset,
    \]
    meaning that there exist $t_0\geq0$ and $n_0\in\N$ so that $\gamma(t_0)\in\Delta_\H(z_n,\rho_\H(z_n,\pi_n))$, for all $n\geq n_0$, contradicting the fact that $\pi_n$ are projections of $z_n$.
\end{proof}

\medskip

Note that by Example \ref{ex: tangential slope example}, the conclusion of Lemma \ref{lm:convergence of projections to zeta} fails if we assume that $\mathrm{arg}(\gamma(t))$ converges to $\pm\tfrac{\pi}{2}$. However, using the same proof, one could reduce the assumption of Lemma \ref{lm:convergence of projections to zeta} to: ``the slope of $\gamma$ contains an angle $\theta\in(-\tfrac{\pi}{2}, \tfrac{\pi}{2})$", recalling that the slope of $\gamma$ is the cluster set of $\arg(\gamma(t))$.

\medskip

Before moving on, we state the following handy fact, whose proof is a simple exercise in calculus.

\begin{lm}\label{lm:cosh fact}
    If $\{x_n\},\{y_n\}\subset \R$ are real sequences with $\lim\limits_{n\to+\infty}(x_n-y_n)=0$, then 
    \[
    \lim_{n\to+\infty} \frac{\cosh x_n}{\cosh y_n}=1.
    \]
\end{lm}

The next theorem is the main result of this section and shows that projections of a sequence onto different curves landing with the same angle are asymptotically close.

\begin{theorem}\label{lm:orthogonal lemma 2}
	Let $\gamma_j:[0,+\infty)\to\H$, $j=1,2$, be two curves landing at $\infty$, with 
    \[
    \lim_{t\to+\infty}\mathrm{arg}(\gamma_j(t))=\theta,
    \]
    for some $\theta\in(-\tfrac{\pi}{2},\tfrac{\pi}{2})$ and $j=1,2$. Let $\{z_n\}\subset\H$ be a sequence converging to $\infty$. If $\pi_{\gamma_j}(z_n)$ is a projection of $z_n$ on $\gamma_j$, for $j=1,2$ and $n\in\N$, then
    \[
    \lim_{n\to+\infty}d_\H(\pi_{\gamma_1}(z_n),\pi_{\gamma_2}(z_n))=0.
    \]
\end{theorem}
\begin{proof}
	In order to simplify notation, we write $\pi_n^j\vcentcolon=\pi_{\gamma_j}(z_n)$ for $j=1,2$ and all $n\in\N$. Note that Lemma \ref{lm:convergence of projections to zeta} is applicable and yields that both sequences $\{\pi_n^1\},\{\pi_n^2\}$ converge to $\infty$, as $n\to+\infty$.\\
    Consider the straight half-line $\gamma\colon[1,+\infty)\to\H$, with $\gamma(t)=te^{i\theta}$. Recall from Lemma \ref{lm: sectors in H} that for all $R>0$, we have that
    \[
    S_R\vcentcolon= \{z\in\H\colon d_\H(z,\gamma)<R\}= D_\H(e^{i\theta},R)\cup\{re^{i\phi}\colon r>1, \ \phi\in(\phi_1,\phi_2)\},
    \]
    for some $\phi_1\in(-\tfrac{\pi}{2},\theta)$ and $\phi_2\in(\theta,\tfrac{\pi}{2})$ satisfying $d_\H(e^{i\theta},e^{i\phi_1})=d_\H(e^{i\theta},e^{i\phi_2})=R$. So, since $\gamma_1$ lands at $\infty$ with $\lim_{t\to+\infty}\arg(\gamma_1(t))=\theta$, we get that for every $R>0$, there exists a $t_0\geq0$, so that $\gamma_1([t_0,+\infty))\subset S_R$. So, for any sequence $\{\zeta_n\}\subset \gamma_1([0,+\infty))$, converging to $\infty$, we have that
    \begin{equation}\label{eq:orthogonal speed lemma 2, eq1}
        \lim_{n\to+\infty}d_\H(\zeta_n,\gamma)=0.
    \end{equation}
    Observe that due to Lemma \ref{lm: properties of metric in H} (1), the unique projections of $z_n$ and $\pi_n^1$ onto $\gamma$ are $\lvert z_n\rvert e^{i\theta}$ and $\lvert \pi_n^1\rvert e^{i\theta}$, respectively, for sufficiently large $n\in\N$ (see Figure \ref{fig:projection}). Since $\{\pi_n^1\}\subset\gamma_1([0,+\infty))$ and $\pi^1_n\xrightarrow{n\to+\infty}\infty$, \eqref{eq:orthogonal speed lemma 2, eq1} yields
    \begin{equation}\label{eq:orthogonal speed lemma 2, eq2}
    \lim_{n\to+\infty}d_\H(\pi^1_n, \lvert \pi^1_n \rvert e^{i\theta})=0.
    \end{equation}
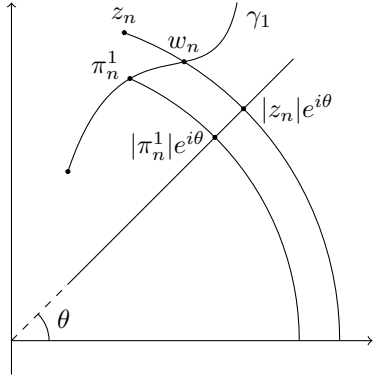
\begin{figure}[ht]
        \centering
        \begin{tikzpicture}[scale=1.5]
        \coordinate (r) at (1,0);
        \coordinate (o) at (0,0);
        \coordinate (theta) at (1,1);
            \draw[->] (0,-0.3) -- (0,3);
            \draw[->] (0,0) -- (3.2,0);
            \node[below right] at (2,3) {$\gamma_1$};
            \draw[dashed] (0,0) -- (0.5,0.5);
            \draw (0.5,0.5) -- (2.5,2.5);
            \draw (0.5, 1.5)  .. controls (1,3) and (1.8,2) .. (2,3);
            
            \filldraw[black] (0.5,1.5) circle (.5pt);
            \filldraw[black] (2.057,2.057) circle (.5pt) node[right, xshift=3pt] {$\lvert z_n \rvert e^{i\theta}$};
            \draw [domain=0:69.9] plot ({2.909*cos(\x)}, {2.909*sin(\x)});
            \filldraw[black] (1,2.732) circle (.5pt) node[above] {$z_n$};
            \filldraw[black] (1.803,1.803) circle (.5pt) node[left, yshift=-3pt] {$\lvert \pi^1_n\rvert e^{i\theta}$};
            \filldraw[black] (1.53,2.474) circle (.5pt) node[above] {$w_n$};
            \filldraw[black] (1.05,2.325) circle (.5pt) node[left, yshift=5pt] {$\pi^1_n$};
            \draw [domain=0:65.695] plot ({2.551*cos(\x)}, {2.551*sin(\x)});
            
            \pic [draw, "$\theta$", angle eccentricity=1.5] {angle = r--o--theta};
        \end{tikzpicture}
        \caption{The projections in Theorem \ref{lm:orthogonal lemma 2}.}
        \label{fig:projection}
    \end{figure}
    
    Now, since, for all $n\in\N$ large, $\lvert z_n \rvert e^{i\theta}$ is the projection of $z_n$ onto $\gamma$ we get that
    \begin{equation}\label{eq:orthogonal speed lemma 2, eq3}
        d_\H(z_n,\lvert z_n \rvert e^{i\theta}) - d_\H(z_n,\lvert \pi^1_n \rvert e^{i\theta})\leq0.
    \end{equation}
    Let $w_n$ be a point of intersection of the geodesic $\{z\in\H \colon \lvert z \rvert = \lvert z_n \rvert\}$ with the trace $\gamma_1([0,+\infty))$, which is well-defined for all $n\in\N$ large. Note that $w_n$ tends to $\infty$, as $n\to+\infty$, and the projection of $w_n$ onto $\gamma$ is $\lvert z_n \rvert e^{i\theta}$. So, applying \eqref{eq:orthogonal speed lemma 2, eq1} once more yields
    \begin{equation}\label{eq:orthogonal speed lemma 2, eq4}
        \lim_{n\to+\infty}d_\H(w_n,\lvert z_n \rvert e^{i\theta})=0.
    \end{equation}
    Using the triangle inequality we have
    \begin{align*}
        d_\H(z_n,\lvert z_n \rvert e^{i\theta}) &\geq d_\H(z_n,w_n) - d_\H(w_n,\lvert z_n \rvert e^{i\theta})\\
            &\geq d_\H(z_n,\pi^1_n) - d_\H(w_n,\lvert z_n \rvert e^{i\theta}) \qquad \text{($\pi^1_n$ is the projection of $z_n$ onto $\gamma_1$)}\\
            &\geq d_\H(z_n, \lvert \pi^1_n \rvert e^{i\theta}) -d_\H(\lvert \pi^1_n \rvert e^{i\theta}, \pi^1_n) - d_\H(w_n,\lvert z_n \rvert e^{i\theta}),
    \end{align*}
    which can be rewritten as 
    \begin{equation}\label{eq:orthogonal speed lemma 2, eq5}
        d_\H(z_n,\lvert z_n \rvert e^{i\theta}) - d_\H(z_n, \lvert \pi^1_n \rvert e^{i\theta}) \geq -d_\H(\lvert \pi^1_n \rvert e^{i\theta}, \pi^1_n)- d_\H(w_n,\lvert z_n \rvert e^{i\theta}).
    \end{equation}
    Combining \eqref{eq:orthogonal speed lemma 2, eq2}, \eqref{eq:orthogonal speed lemma 2, eq3}, \eqref{eq:orthogonal speed lemma 2, eq4} and \eqref{eq:orthogonal speed lemma 2, eq5} yields 
    \begin{equation}\label{eq:orthogonal speed lemma 2, eq6}
    \lim_{n\to+\infty}\left(d_\H(z_n,\lvert z_n \rvert e^{i\theta}) - d_\H(z_n, \lvert \pi^1_n \rvert e^{i\theta})\right) =0.
    \end{equation}
   Our goal now is to show that
    \begin{equation}\label{eq:orthogonal speed lemma 2, eq7}
        \lim_{n\to+\infty} d_\H(\lvert \pi^1_n\rvert e^{i\theta}, \lvert z_n\rvert e^{i\theta})=0.
    \end{equation}
    Note that it suffices to assume that $\lvert z_n \rvert \neq \lvert \pi^1_n \rvert$, for all $n\in\N$. Moreover, by splitting $\{z_n\}$ and $\{\pi^1_n\}$ into two subsequences, we can assume that $\lvert z_n \rvert > \lvert \pi^1_n \rvert$ for all $n\in\N$; the proof for the subsequence satisfying $\lvert z_n \rvert < \lvert \pi^1_n \rvert$ is similar.\\
    The limit \eqref{eq:orthogonal speed lemma 2, eq6} implies that Lemma \ref{lm:cosh fact} is applicable to the real sequences $\{d_\H(z_n,\lvert z_n \rvert e^{i\theta})\}$ and $\{d_\H(z_n, \lvert \pi^1_n \rvert e^{i\theta})\}$, and yields
    \begin{equation}\label{eq:orthogonal speed lemma 2, eq8}
        \lim_{n\to+\infty}\frac{\cosh d_\H(z_n, \lvert \pi^1_n \rvert e^{i\theta})}{\cosh d_\H(z_n,\lvert z_n \rvert e^{i\theta})}=1.
    \end{equation}
    Write $\phi_n\vcentcolon=\mathrm{arg} \ z_n\in(-\tfrac{\pi}{2},\tfrac{\pi}{2})$. Using \eqref{eq:cosh in H} we obtain
    \[
        \frac{\cosh d_\H(z_n, \lvert \pi^1_n \rvert e^{i\theta})}{\cosh d_\H(z_n,\lvert z_n \rvert e^{i\theta})}=\frac{\lvert z_n  + \lvert \pi^1_n \rvert e^{-i\theta}\rvert}{\lvert z_n  + \lvert z_n \rvert e^{-i\theta}\rvert}\frac{\sqrt{\mathrm{Re}\ \lvert z_n \rvert e^{i\theta}}}{\sqrt{\mathrm{Re}\ \lvert \pi^1_n \rvert e^{i\theta}}} =\frac{\left\lvert e^{i(\phi_n+\theta)} + \left\lvert \frac{\pi^1_n}{z_n}\right\rvert\right\rvert}{\left\lvert e^{i(\phi_n+\theta)} + 1\right\rvert} \frac{1}{\sqrt{\left\lvert \frac{ \pi^1_n }{z_n}\right\rvert}}.
    \]
    Taking limits  as $n$ tends to $+\infty$, and using \eqref{eq:orthogonal speed lemma 2, eq8}, we obtain
    \begin{equation} \label{eq:orthogonal speed lemma 2, eq9}
    \lim_{n\to+\infty}\frac{\left\lvert e^{i(\phi_n+\theta)} + \left\lvert \frac{\pi^1_n}{z_n}\right\rvert\right\rvert}{\left\lvert e^{i(\phi_n+\theta)} + 1\right\rvert} \frac{1}{\sqrt{\left\lvert \frac{ \pi^1_n }{z_n}\right\rvert}}=1.
    \end{equation}
    Due to our assumption we have that $\left\lvert \frac{ \pi^1_n }{z_n}\right\rvert\in(0,1)$, for all $n\in\N$. We are going to show that 
    \begin{equation}\label{eq:orthogonal speed lemma 2, eq10}
        \lim_{n\to+\infty}\left\lvert \frac{ \pi^1_n }{z_n}\right\rvert=1.
    \end{equation}
    Let $x_0\in[0,1]$ be an accumulation point of $\left\lvert \frac{ \pi^1_n }{z_n}\right\rvert$. By passing to a subsequence if necessary we suppose that $\{\phi_n+\theta\}$ converges to some $\phi\in(-\pi,\pi)$; the fact that $\phi\neq\pm\pi$ follows from $\theta\in(-\tfrac{\pi}{2},\tfrac{\pi}{2})$. Observe that $x_0\neq 0$, because otherwise 
    \[
    \lim_{n\to+\infty}\frac{1}{\sqrt{\left\lvert \frac{ \pi^1_n }{z_n}\right\rvert}}=+\infty,
    \]
    while 
    \[
    \lim_{n\to+\infty}\frac{\left\lvert e^{i(\phi_n+\theta)} + \left\lvert \frac{\pi^1_n}{z_n}\right\rvert\right\rvert}{\left\lvert e^{i(\phi_n+\theta)} + 1\right\rvert} = \frac{1}{\lvert e^{i\phi} +1\rvert }\geq \frac{1}{2},
    \]
    contradicting the finiteness of the limit in \eqref{eq:orthogonal speed lemma 2, eq9}. Thus, \eqref{eq:orthogonal speed lemma 2, eq9} can be rewritten as
    \[
    \frac{\left\lvert e^{i\phi} + x_0\right\rvert}{\left\lvert e^{i\phi} + 1\right\rvert} \frac{1}{\sqrt{x_0}}=1,
    \]
    which in turn is equivalent to
    \[
    \frac{(\cos\phi +x_0)^2+\sin^2\phi}{(\cos\phi +1)^2+\sin^2\phi} = x_0 \iff (x_0-1)^2=0.
    \]
    So, 1 is the only accumulation point of the bounded sequence  $\left\{\left\lvert \frac{ \pi^1_n }{z_n}\right\rvert\right\}$, proving \eqref{eq:orthogonal speed lemma 2, eq10}. Now, the formula \eqref{eq:hyperbolic distance in H} for the hyperbolic metric of $\H$ implies that 
    \[
    \lim_{n\to+\infty}d_\H(\lvert z_n \rvert , \lvert \pi^1_n \rvert)=\lim_{n\to+\infty}\frac{1}{2}\left\lvert \log\left\lvert \frac{ \pi^1_n }{z_n}\right\rvert\right\rvert=0.
    \]
    So, Lemma \ref{lm:sequences lemma in H} applied to the sequences $\{\lvert z_n \rvert e^{i\theta}\}$ and $\{\lvert \pi^1_n\rvert e^{i\theta}\}$ yields \eqref{eq:orthogonal speed lemma 2, eq7}.\\
    With \eqref{eq:orthogonal speed lemma 2, eq7} proved, we use the triangle inequality to obtain
    \begin{equation}\label{eq:orthogonal speed lemma 2, eq+}
    d_\H(\lvert z_n \rvert e^{i\theta}, \pi^1_n)\leq d_\H(\lvert z_n \rvert e^{i\theta}, \lvert \pi^1_n \rvert e^{i\theta}) + d_\H(\lvert \pi^1_n \rvert e^{i\theta}, \pi^1_n)
    \end{equation}
    But according to \eqref{eq:orthogonal speed lemma 2, eq2} and \eqref{eq:orthogonal speed lemma 2, eq7}, the right-hand side term in \eqref{eq:orthogonal speed lemma 2, eq+} tends to zero, as $n\to+\infty$. Hence,
    \begin{equation}\label{eq:orthogonal speed lemma 2, eq11}
        \lim_{n\to+\infty}d_\H(\lvert z_n \rvert e^{i\theta}, \pi^1_n)=0.
    \end{equation}
    All of the above arguments can be applied verbatim to the curve $\gamma_2$ and the projections $\{\pi^2_n\}$ in order to show that 
    \begin{equation}\label{eq:orthogonal speed lemma 2, eq12}
        \lim_{n\to+\infty}d_\H(\lvert z_n \rvert e^{i\theta}, \pi^2_n)=0.
    \end{equation}
    Finally, the limits \eqref{eq:orthogonal speed lemma 2, eq11} and \eqref{eq:orthogonal speed lemma 2, eq12} along with a simple use of the triangle inequality yield
    \[
    \lim_{n\to+\infty}d_\H(\pi^1_n, \pi^2_n)=0,
    \]
    as required.
\end{proof}

Using Theorem \ref{lm:orthogonal lemma 2} we can correlate the projections of a sequence onto curves that land with different angles.

\begin{cor}\label{cor:projections with different slopes}
Let $\gamma_j:[0,+\infty)\to\H$, $j=1,2$, be two curves landing at $\infty$, with 
\[
\lim\limits_{t\to+\infty}\mathrm{arg}(\gamma_j(t))=\theta_j,
\]
for some $\theta_j\in(-\tfrac{\pi}{2},\tfrac{\pi}{2})$ and $j=1,2$. Suppose that $\{z_n\}\subset\H$ is a sequence converging to $\infty$ and let $\pi_{\gamma_j}(z_n)$ be a projection of $z_n$ in $\gamma_j$, for $j=1,2$ and $n\in\N$. Then,
    \[
    \lim_{n\to+\infty}d_\H(\pi_{\gamma_1}(z_n),\pi_{\gamma_2}(z_n))=d_\H(e^{i\theta_1},e^{i\theta_2}).
    \]
\end{cor}

\begin{proof}
    Consider the straight half-lines $\eta_1\colon[1,+\infty)\to\H$ and $\eta_2\colon[1,+\infty)\to\H$, with 
    \[
    \eta_1(t) = te^{i\theta_1} \quad\text{and}\quad \eta_2(t) = te^{i\theta_2}, \quad \text{for}\ t\in[1,+\infty).
    \]
    Note that the projections of $z_n$ onto $\eta_1$ and $\eta_2$ are $\lvert z_n \rvert e^{i\theta_1}$ and $\lvert z_n \rvert e^{i\theta_2}$, respectively, for all $n$ large. Moreover, from Lemma \ref{lm: properties of metric in H} (2), for all $n\in\N$ we have $d_\H(\lvert z_n \rvert e^{i\theta_1}, \lvert z_n \rvert e^{i\theta_2}) =d_\H(e^{i\theta_1}, e^{i\theta_2})$. Then, applying Theorem \ref{lm:orthogonal lemma 2} to the pairs of curves $\gamma_1,\eta_1$ and $\gamma_2,\eta_2$, we have that 
    \begin{equation}\label{eq:projections with different slopes, eq1}
        \lim_{n\to+\infty}d_\H(\pi_{\gamma_1}(z_n),\lvert z_n \rvert e^{i\theta_1}) = \lim_{n\to+\infty}d_\H(\pi_{\gamma_2}(z_n),\lvert z_n \rvert e^{i\theta_2}) = 0.
    \end{equation}
    Using the triangle inequality, we obtain
    \begin{align*}
        d_\H(\lvert z_n \rvert e^{i\theta_2}, \pi_{\gamma_1}(z_n))&\geq d_\H(\lvert z_n \rvert e^{i\theta_2}, \lvert z_n \rvert e^{i\theta_1}) - d_\H(\lvert z_n \rvert e^{i\theta_1}, \pi_{\gamma_1}(z_n))\\
        &=d_\H( e^{i\theta_2},  e^{i\theta_1}) - d_\H(\lvert z_n \rvert e^{i\theta_1}, \pi_{\gamma_1}(z_n)),
    \end{align*}
    and similarly
    \[
    d_\H(\lvert z_n \rvert e^{i\theta_2}, \pi_{\gamma_1}(z_n)) \leq d_\H(e^{i\theta_2},  e^{i\theta_1}) + d_\H(\lvert z_n \rvert e^{i\theta_1}, \pi_{\gamma_1}(z_n)).
    \]
    So, \eqref{eq:projections with different slopes, eq1} implies that 
    \begin{equation}\label{eq:projections with different slopes, eq2}
        \lim_{n\to+\infty}d_\H(\lvert z_n \rvert e^{i\theta_2}, \pi_{\gamma_1}(z_n)) = d_\H(e^{i\theta_1}, e^{i\theta_2}).
    \end{equation}
    Another use of the triangle inequality yields
    \[
    d_\H(\pi_{\gamma_1}(z_n), \pi_{\gamma_2}(z_n)) \geq d_\H(\pi_{\gamma_1}(z_n), \lvert z_n \rvert e^{i\theta_2}) - d_\H(\lvert z_n \rvert e^{i\theta_2}, \pi_{\gamma_2}(z_n)),
    \]
    and 
    \[
    d_\H(\pi_{\gamma_1}(z_n), \pi_{\gamma_2}(z_n)) \leq  d_\H(\pi_{\gamma_1}(z_n), \lvert z_n \rvert e^{i\theta_2}) + d_\H(\lvert z_n \rvert e^{i\theta_2}, \pi_{\gamma_2}(z_n)).
    \]
    Using \eqref{eq:projections with different slopes, eq1} and \eqref{eq:projections with different slopes, eq2} we get
    \[
    \lim_{n\to\infty}d_\H(\pi_{\gamma_1}(z_n), \pi_{\gamma_2}(z_n)) = d_\H(e^{i\theta_1}, e^{i\theta_2}),
    \]
    and the proof is complete. 
\end{proof}

\section{Proofs of the monotonicity properties}\label{sect: monotonicity}

This section contains the proofs of our main result, Theorem \ref{thm: main result}, and its corollary, Corollary \ref{cor: main cor}. We start with a simple lemma that uses the angular derivative.

\begin{lm}\label{lm:orthogonal speed lemma}
	Let $f:\H\to\H$ be a hyperbolic map with Denjoy--Wolff point $\infty$. For any $z\in\H$ and $w\in\C$ the sequence $\{\lvert f^n(z) - w\rvert\}$ is eventually strictly increasing.
\end{lm}

\begin{proof}
    Because $f$ is hyperbolic, the angular derivative condition \eqref{eq: angular derivative} implies that
    \[
    f'(\infty)=\angle\lim_{z\to\infty}\frac{f(z)}{z}>1.
    \]
    Now, fix $z\in\H$ and $w\in\C$. Theorem \ref{thm: step and slope for hyperbolic maps} (2) implies that the sequence $\{f^n(z)\}$ tends to infinity through a sector $\{z\in\C\colon \lvert \mathrm{arg}\ z\rvert<\phi\}$, for some $\phi\in(-\tfrac{\pi}{2},\tfrac{\pi}{2})$. So, the above limit can be rewritten as
    \[
    \lim_{n\to+\infty}\left \lvert \frac{f(f^n(z)) - w}{f^n(z) -w}\right\rvert =\lim_{n\to+\infty}\left \lvert \frac{f(f^n(z))}{f^n(z)}\right\rvert= \angle\lim_{z\to\infty}\left \lvert \frac{f(z)}{z}\right \rvert  >1,
    \]
    which completes the proof.
\end{proof}

Using Lemma~\ref{lm:orthogonal speed lemma} we can prove that projections of hyperbolic orbits onto geodesics are eventually strictly increasing, as stated in the following corollary. Note that this result is merely a special case of our main theorem. We include a proof, however, to highlight the fact that the monotonicity of the orthogonal speed of hyperbolic maps does not require the heavy machinery of Section \ref{sect:projections}.

\begin{cor}\label{cor:orthogonal speed corollary}
	Let $f:\H\to\H$ be a hyperbolic map with Denjoy--Wolff point $\infty$. For any $w\in\H$, consider the hyperbolic geodesic $\gamma$ of $\H$ emanating from $w$ and landing at $\infty$. If $z\in\H$ and $\{\pi_\gamma(f^n(z))\}$ denotes the sequence of projections of $f^n(z)$ onto $\gamma$, then the sequence $\{d_\H(w,\pi_\gamma(f^n(z)))\}$ is eventually strictly increasing.
\end{cor}

\begin{proof}
    Fix $z,w\in\H$. Note that the geodesic $\gamma\colon[0,+\infty)\to\H$ joining $w$ and $\infty$ is the horizontal half-line $\gamma(t)=w+t$. Lemma \ref{lm:orthogonal speed lemma} is applicable and yields that the sequence $\{\lvert f^n(z) - i\ \mathrm{Im}\ w\rvert \}$ is eventually strictly increasing. Notice that the projection of $f^n(z)$ onto $\gamma$ is the point 
\[
\pi_\gamma(f^n(z))=\lvert f^n(z) - i\ \mathrm{Im}\ w\rvert + i\ \mathrm{Im}w,
\]
for all large $n\in\N$. So, using a vertical translation and the invariance of the hyperbolic metric under M\"obius maps we obtain
\[
d_\H(w,\pi_\gamma(f^n(z))) =d_\H(w, \lvert f^n(z) - i\ \mathrm{Im}\ w\rvert + i\ \mathrm{Im}w)=d_\H(\mathrm{Re}w, \lvert f^n(z) - i\ \mathrm{Im}\ w\rvert).
\]
So the result follows from the monotonicity of $\{\lvert f^n(z) - i\ \mathrm{Im}\ w\rvert \}$ and Lemma \ref{lm: properties of metric in H} (1).
\end{proof}

Using Lemma \ref{lm:orthogonal speed lemma} and the main result of Section \ref{sect:projections}, we can now prove Theorem \ref{thm: main result}. We are going to prove its equivalent formulation in $\H$, stated in Theorem \ref{thm: main monotonicity result} and restated below for the convenience of the reader. 

\begin{theo*}
    Let $f\colon\H\to\H$ be a hyperbolic map with Denjoy--Wolff point $\infty$, and $\gamma\colon [0,+\infty)\to\H$ a curve landing at $\infty$, satisfying
    \[
    \lim_{t\to+\infty}\mathrm{arg}(\gamma(t))=\theta\in\left(-\tfrac{\pi}{2},\tfrac{\pi}{2}\right).
    \]
    Fix $z\in\D$ and let $\pi_\gamma(f^n(z))$ be a projection of $f^n(z)$ onto $\gamma$, for $n\in\N$. Then, for any $w\in\H$, the sequence $\{d_\H(w,\pi_\gamma(f^n(z)))\}$ is eventually strictly increasing.
\end{theo*}

\begin{proof}
    Fix $z,w \in \H$, and consider the straight half-line $\eta\colon[0,+\infty)\to\H$, with $\eta(t)=w+te^{i\theta}$. It is easy to see that the vertical translation $z\mapsto z-i\left(\mathrm{Im}\ w - \tan\theta\mathrm{Re}\ w\right)$ maps $\eta([0,+\infty))$ to the half-line 
    \[
    \left\{te^{i\theta}\colon t\geq \frac{\mathrm{Re}\ w}{\cos\theta}\right\}.
    \]
    Recall that by Lemma \ref{lm: properties of metric in H} the projection of a point $z_0\in\H$ onto the half-line $\{te^{i\theta}\colon t>0\}$ is $\lvert z_0\rvert e^{i\theta}$. So, the projection of $f^n(z)$ onto $\eta$ is the point 
    \begin{equation}\label{eq:orthogonal projection, eq0}
    \pi_\eta(f^n(z))=\lvert f^n(z)-i\left(\mathrm{Im}\ w - \tan\theta\mathrm{Re}\ w\right)\rvert e^{i\theta} + i\left(\mathrm{Im}\ w - \tan\theta\mathrm{Re}\ w\right).
    \end{equation}
    To simplify our notation, we write $\pi_\eta(n)=\pi_\eta(f^n(z))$ and $\pi_\gamma(n)=\pi_\gamma(f^n(z))$, for all $n\in\N$. Our goal is to show that $\{d_\H(w,\pi_\gamma(n)\}$ is eventually strictly increasing. Observe that is suffices to show that 
    \begin{equation}\label{eq:orthogonal projection, eq+}
        \liminf_{n\to+\infty} \left(d_\H(w,\pi_\gamma(n+1)) - d_\H(w,\pi_\gamma(n))\right) >0.
    \end{equation}
    By the triangle inequality we get that
    \begin{align*}
    &d_\H(w,\pi_\gamma(n+1)) - d_\H(w,\pi_\gamma(n)) \geq \\
    &\geq d_\H(w,\pi_\eta(n+1)) - d_\H(\pi_\eta(n+1), \pi_\gamma(n+1)) - d_\H(w,\pi_\eta(n)) - d_\H(\pi_\eta(n), \pi_\gamma(n))\\
    &= d_\H(w,\pi_\eta(n+1)) - d_\H(w,\pi_\eta(n)) -d_\H(\pi_\eta(n+1), \pi_\gamma(n+1)) - d_\H(\pi_\eta(n), \pi_\gamma(n)).
    \end{align*}
    Because $\lim\limits_{t\to+\infty}\mathrm{arg}(\eta(t))=\lim\limits_{t\to+\infty}\mathrm{arg}(\gamma(t))=\theta$, Theorem \ref{lm:orthogonal lemma 2} is applicable and yields that
    \[
    \lim_{n\to+\infty}d_\H(\pi_\eta(n),\pi_\gamma(n))=\lim_{n\to+\infty}d_\H(\pi_\eta(n+1),\pi_\gamma(n+1))=0.
    \]
    So, in order to prove our goal, \eqref{eq:orthogonal projection, eq+}, it suffices to show that 
    \begin{equation}\label{eq:orthogonal projection, eq1}
        \liminf_{n\to+\infty} \left(d_\H(w,\pi_\eta(n+1))- d_\H(w,\pi_\eta(n))\right)>0.
    \end{equation}
    Write 
    \begin{equation}\label{eq:orthogonal projection, eq-}
        f^n(z)-i(\mathrm{Im}\ w - \tan\theta\mathrm{Re}\ w)=r_ne^{i\theta_n},
    \end{equation}
    and 
    \begin{equation}\label{eq:orthogonal projection, eq--}
        w_0=w-i\left(\mathrm{Im} \ w -\tan\theta\ \mathrm{Re}\ w\right)=\mathrm{Re}w+i \tan\theta\ \mathrm{Re}w.
    \end{equation}
    See Figure \ref{fig:main monotonicity result} for the construction. Using the translation $z\mapsto z- i(\mathrm{Im}\ w - \tan\theta\mathrm{Re}\ w)$, the invariance of the hyperbolic metric under M\"obius maps and \eqref{eq:orthogonal projection, eq0}, we obtain
    \[
        d_\H(w,\pi_\eta(f^n(z))) =d_\H(w_0, r_ne^{i\theta}).
    \]
    So, \eqref{eq:orthogonal projection, eq1} can be rewritten as
    \begin{equation}\label{eq:orthogonal projection, eq++}
       \liminf_{n\to+\infty} \left(d_\H(w_0, r_{n+1} e^{i\theta} ) - d_\H(w_0, r_n e^{i\theta} )\right)>0.
    \end{equation}
    \begin{figure}[ht]
        \centering
        \begin{tikzpicture}[scale=1.5]
        \coordinate (r) at (1,0);
        \coordinate (o) at (0,0);
        \coordinate (theta) at (1,1);
            \draw[->] (0,-0.3) -- (0,3);
            \draw[->] (0,0) -- (3,0);
            \draw   (0.5,1.5) -- (2,3);
            \node[below] at (2,3) {$\eta$};
            \draw[dashed] (0,0) -- (1.5,1.5);
            \draw[dotted] (0.5,1.5) -- (.5,.5);
            \draw[dotted] (1.414,2.414) -- (1.414,1.414);
            
            \filldraw[black] (0.5,1.5) circle (.5pt) node[left] {$w$};
            \filldraw[black] (.5,.5) circle (.5pt) node[above left] {$w_0$};
            \filldraw[black] (1.414,2.414) circle (.5pt) node[right] {$\pi_\eta(f^n(z))$};
            \draw [domain=45:60] plot ({2*cos(\x)}, {2*sin(\x)+1});
            \filldraw[black] (1,2.732) circle (.5pt) node[above] {$f^n(z)$};
            \filldraw[black] (1.414,1.414) circle (.5pt) node[right] {$r_ne^{i\theta}$};
            
            \pic [draw, "$\theta$", angle eccentricity=1.5] {angle = r--o--theta};
        \end{tikzpicture}
        \caption{The framework for the proof of Theorem \ref{thm: main monotonicity result}.}
        \label{fig:main monotonicity result}
    \end{figure}
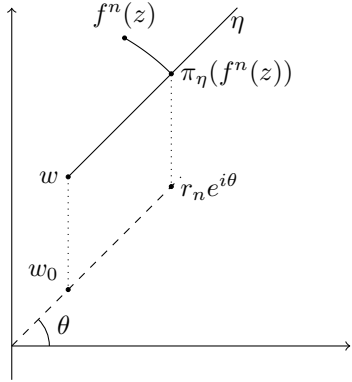
    Now, recalling \eqref{eq:orthogonal projection, eq-}, Lemma \ref{lm:orthogonal speed lemma} implies that the sequence 
    \[
    \{r_n\}=\{\lvert f^n(z)-i(\mathrm{Im}\ w - \tan\theta\mathrm{Re}\ w)\rvert\},
    \]
    is eventually strictly increasing, and $r_n\to+\infty$ as $n\to+\infty$.\\
    Observe that due to \eqref{eq:orthogonal projection, eq--} the point $w_0$ lies in the half-line $\{te^{i\theta}\colon t>0\}$. So, $w_0=\lvert w_0 \rvert e^{i\theta}$. This means that we can apply Lemma \ref{lm: differences lemma} to the sequence $\{r_n\}$ to obtain that 
    \begin{equation}\label{eq:orthogonal projection, eq extra1}
        \lim_{n\to+\infty} \left(d_\H(w_0,r_ne^{i\theta}) -d_\H(\lvert w_0 \rvert, r_n)\right)=-\log(\cos\theta)\in[0,+\infty).
    \end{equation}
    We also have that 
    \begin{align*}
        d_\H(w_0,r_{n+1}e^{i\theta}) &- d_\H(w_0,r_ne^{i\theta})=\\
        &= d_\H(w_0,r_{n+1}e^{i\theta}) - d_\H(\lvert w_0 \rvert, r_{n+1}) -d_\H(w_0,r_ne^{i\theta}) +d_\H(\lvert w_0 \rvert, r_n) +\\
        &\ \ \   + d_\H(\lvert w_0 \rvert, r_{n+1}) - d_\H(\lvert w_0 \rvert, r_n).
    \end{align*}
    Therefore, applying \eqref{eq:orthogonal projection, eq extra1} to the above equation, we get
    \[
    \liminf_{n\to+\infty}\left( d_\H(w_0,r_{n+1}e^{i\theta}) - d_\H(w_0,r_ne^{i\theta})\right) = \liminf_{n\to+\infty} \left(d_\H(\lvert w_0 \rvert, r_{n+1}) - d_\H(\lvert w_0 \rvert, r_n)\right).
    \]
    So, \eqref{eq:orthogonal projection, eq++} is equivalent to 
    \begin{equation}\label{eq:orthogonal projection, eq extra2}
        \liminf_{n\to+\infty} \left(d_\H(\lvert w_0 \rvert, r_{n+1}) - d_\H(\lvert w_0 \rvert, r_n)\right)>0.
    \end{equation}
    Summarizing the work in this proof so far, we have that in order to show the desired result it suffices to establish \eqref{eq:orthogonal projection, eq extra2}.\\
    Because $\{r_n\}$ is eventually strictly increasing and the points $\lvert w_0 \rvert$, $r_n$ lie on a geodesic of $\H$ (namely, the real axis), we have that 
    \begin{equation}\label{eq:orthogonal projection, eq extra3}
        d_\H(\lvert w_0 \rvert, r_{n+1}) - d_\H(\lvert w_0 \rvert, r_n) = d_\H(r_{n+1}, r_n), \quad \text{for all $n$ large}.
    \end{equation}
    Moreover, since $f$ is hyperbolic, Theorem \ref{thm: step and slope for hyperbolic maps} implies that there exist $d_0>0$ and $\phi\in(-\tfrac{\pi}{2},\tfrac{\pi}{2})$ so that
    \[
    \lim_{n\to+\infty}d_\H(r_{n+1}e^{i\theta_{n+1}}, r_ne^{i\theta_n}) = \lim_{n\to+\infty} d_\H(f^{n+1}(z), f^n(z))=d_0,
    \]
    and
    \[
    \lim_{n\to+\infty}\theta_n=\lim_{n\to+\infty}\mathrm{arg}(r_ne^{i\theta_n})= \lim_{n\to+\infty}\mathrm{arg}(f^n(z))=\phi.
    \]
    These last two limits imply that Lemma \ref{lm:sequences lemma in H} is applicable to the sequence $\{r_ne^{i\theta}\}$ and yields a constant $d>0$ so that 
    \begin{equation}\label{eq:orthogonal projection, eq extra4}
        \lim_{n\to+\infty}d_\H(r_{n+1},r_n)=d.
    \end{equation}
    Combining \eqref{eq:orthogonal projection, eq extra3} with \eqref{eq:orthogonal projection, eq extra4} yields \eqref{eq:orthogonal projection, eq extra2}, as required. 
   \end{proof}

   With our main result proved, we move on to its corollary, Corollary \ref{cor: main cor}. As usual, we prove its right half-plane version, Corollary \ref{cor: main cor H version}, restated below.

   \begin{cor*}
       Let $f\colon\H\to\H$ be a hyperbolic map with Denjoy--Wolff point $\infty$. For any $z,w\in\H$, the sequence $\{d_\H(w,f^n(z))\}$ is eventually strictly increasing.
   \end{cor*}

    \begin{proof}
        Fix $z,w\in\H$. Since $f$ is hyperbolic, Theorem \ref{thm: step and slope for hyperbolic maps} (2) implies that there exists $\phi\in(-\tfrac{\pi}{2},\tfrac{\pi}{2})$, depending only on $f$ and $z$, so that
        \begin{equation}\label{eq: main cor, eq}
            \lim_{n\to+\infty}\arg(f^n(z))=\phi.
        \end{equation}
        Let $\ell_n=[f^n(z),f^{n+1}(z)]$ be the Euclidean line segment joining $f^n(z)$ and $f^{n+1}(z)$, for $n\in\N$. We obviously have that $\ell_n\subset \H$, for all $n\in\N$. Let $\gamma\colon [0,+\infty)\to\H$ be a curve with trace $\gamma([0,+\infty))=\bigcup_{n\in\N}\ell_n$. Since the sequence of iterates $\{f^n(z)\}$ converges to $\infty$, we can parametrise $\gamma$ so that $\lim_{t\to+\infty}\gamma(t)=\infty$. Moreover, \eqref{eq: main cor, eq} implies that $\lim_{t\to+\infty}\arg(\gamma(t))=\phi$. Therefore, $\gamma$ satisfies the assumptions of Theorem \ref{thm: main monotonicity result} and yields that the sequence $\{d_\H(w,\pi_\gamma(f^n(z)))\}$ is eventually strictly increasing. But, by construction $f^n(z)\in\gamma([0,+\infty))$, for all $n\in\N$, meaning that $\pi_\gamma(f^n(z))=f^n(z)$ and the proof is complete.
    \end{proof}

\section{Concluding remarks}\label{sect: concluding remarks}

Recall that in Example \ref{ex: parabolic map} we showed that Theorem \ref{thm: main result} does not hold if one replaces the hyperbolic map with a parabolic. However, we can see that in this particular counterexample, the sequences $\{d_\H(1,\pi_\delta(f^n(1)))\}$ and $\{d_\H(1,\pi_\gamma(g^n(1)))\}$ we constructed for the zero and positive hyperbolic step, respectively, are eventually increasing (but not strictly). It would be interesting to examine whether this behaviour is the norm for parabolic maps, or merely a product of our construction. A first step would be to consider projections of parabolic orbits onto hyperbolic geodesics. In other words, we ask if the orthogonal speed of parabolic maps is eventually increasing. We now present simple arguments that provide a positive answer to this question for parabolic maps of positive hyperbolic step. 

It is known \cite[Theorem 1]{Pom-parabolic} that for a parabolic map of positive hyperbolic step, we have
\begin{equation}\label{eq: positive parabolic}
    \lim_{n\to+\infty}\frac{\mathrm{Im}\ (f^{n+1}(1)) -\mathrm{Im}\ (f^n(1))}{\mathrm{Re}\ (f^n (1))} =b\in\R^*.
\end{equation}
Moreover, the sequence $\{\mathrm{Im}\ (f^n(1))\}$ is either eventually positive, or eventually negative (see \cite[Remark 2.3]{CCZRP 2}), whereas the sequence $\{\mathrm{Re}\ (f^n(1))\}$ is non-decreasing by Julia's Lemma (see \cite[Section 2.1]{Abate}). Combining these two facts with \eqref{eq: positive parabolic}, we obtain that the sequence $\{|\mathrm{Im}\ (f^n(1))|\}$ is eventually strictly increasing, and so the same holds for $\{\lvert f^n(1)\rvert\}$. Hence, if $\gamma\colon[0,+\infty) \to \H$ is the hyperbolic geodesic with $\gamma(t)=1+t$, then $\pi_\gamma(f^n(1))=\lvert f^n(1)\rvert$, meaning that $\{d_\H(1,\pi_\gamma(f^n(1)))$\} is eventually strictly increasing.

Similarly, one could check that in Example \ref{ex: hyperbolic semigroup}, the function $t\mapsto d_\H(1,\pi_\delta(\phi_t(1)))$ is also eventually increasing. Therefore, it might be possible to show that Theorem \ref{thm: main result} holds with a hyperbolic semigroup in place of a the hyperbolic map $f$, and with the conclusion ``eventually increasing" instead of ``eventually strictly increasing".

To discuss a final question that arises from our work, let us first formally define the slope of a curve in $\D$. 

Let $\gamma:[0,+\infty)\to \D$ be a curve landing at a boundary point $\tau \in\partial\D$. The \emph{slope of $\gamma$} is the set $\mathrm{Slope}(\gamma)$ consisting of all $\theta\in[-\tfrac{\pi}{2},\tfrac{\pi}{2}]$, for which there exists a sequence $\{t_n\}\subset[0,+\infty)$ such that $\arg(1-\overline{\zeta}\gamma(t_n))\to\theta$.

Consider a hyperbolic map $f\colon\D\to\D$. With the above terminology, Theorem \ref{thm: main result} yields the eventual monotonicity of the sequence $\{d_\D(w,\pi_\gamma(f^n(z)))\}$ under the assumption that $\mathrm{Slope}(\gamma)=\{\theta\}$, for some $\theta\in(-\tfrac{\pi}{2},\tfrac{\pi}{2})$. Also, Example \ref{ex: tangential slope example} shows that the monotonicity of projections described in Theorem \ref{thm: main result} fails if we assume that $\mathrm{Slope}(\gamma)$ is either $\{-\tfrac{\pi}{2}\}$ or $\{\tfrac{\pi}{2}\}$.

However, as we remarked in Section \ref{sect:projections}, right after Lemma \ref{lm:convergence of projections to zeta}, if we assume that there exists $\theta\in (-\tfrac{\pi}{2},\tfrac{\pi}{2})$ so that $\theta\in\mathrm{Slope}(\gamma)$, then the projections $\pi_\gamma(f^n(z))$ will still converge to $\tau$. Therefore, it makes sense to study the monotonicity of $\{d_\D(w,\pi_\gamma(f^n(z)))\}$ in this broader setting. The key points of our analysis in Section \ref{sect: monotonicity} revolved around the fact that each orbit $\{f^n(z)\}$ tends to $\tau$ with a fixed slope, and that hyperbolic maps have positive step, described in Theorem \ref{thm: step and slope for hyperbolic maps}. These remain intact even in the setting of curves with ``wild" slope. Our main technical result, however, Theorem \ref{lm:orthogonal lemma 2}, cannot be used to substitute the projections on such $\gamma$ with projections on a simpler curve. Finally, preliminary examples also show that the monotonicity of projections in this case, if it indeed holds, will no longer be strict.

\section*{Acknowledgements}
We thank M. D. Contreras, S. D\'iaz--Madrigal and L. Rodr\'{i}guez-Piazza as well as C. Papadimitriou and G. Polychrou for the helpful discussions and comments. We are also grateful to F. Cruz--Zamorano for pointing us to the simple proof of Corollary \ref{cor:orthogonal speed corollary} presented here.

\end{document}